\documentclass[11pt, leqno]{article}

\usepackage[utf8]{inputenc}
\usepackage[T1]{fontenc}
\usepackage{microtype}
\usepackage[a4paper]{geometry}

\usepackage[dvipsnames]{xcolor}

\usepackage{amsmath}
\usepackage{amsthm}
\usepackage{tikz-cd}
\usetikzlibrary{arrows} 
\tikzset{
  commutative diagrams/.cd, 
  arrow style=tikz, 
  diagrams={>=stealth}
}
\usetikzlibrary{matrix,decorations.pathreplacing,calc}

\usepackage{textcomp}
\usepackage[sb]{libertine}
\usepackage[varqu,varl]{zi4}%
\usepackage[libertine,bigdelims]{newtxmath}

\useosf

\usepackage[scr=boondox,cal=euler]{mathalfa}

\usepackage{marvosym}

\usepackage{slashed}
\usepackage{esint} 
\usepackage[english]{babel}

\usepackage{imakeidx}
\indexsetup{noclearpage}
\makeindex[intoc]

\usepackage{csquotes}
\usepackage[
  backend=biber,
  hyperref=true,
  backref=true,
  isbn=false,
  doi=true,
  natbib=true,
  eprint=true,
  useprefix=true,
  maxcitenames=99,
  maxbibnames=99,  
  maxalphanames=99, 
  minalphanames=99,
  safeinputenc,
  style=alphabetic,
  citestyle=alphabetic,
  block=space,
  datamodel=preamble/ext-eprint,
  sorting=nyt
]{biblatex}

\definecolor{darkblue}{HTML}{0000A6}

\usepackage[
  bookmarksnumbered = true,
  hypertexnames     = false,
  colorlinks        = true,
  citecolor         = darkblue,
  linkcolor         = darkblue,
  urlcolor          = darkblue,
  breaklinks
]{hyperref}

\DeclareFieldFormat{url}{%
  \href{#1}{\ComputerMouse}
}
\DeclareFieldFormat{doi}{%
  \mkbibacro{DOI}\addcolon\space\href{https://doi.org/#1}{#1}
}
\makeatletter
\DeclareFieldFormat{arxiv}{%
  arXiv\addcolon\space\href{http://arxiv.org/\abx@arxivpath/#1}{#1}
}
\makeatother
\DeclareFieldFormat{mr}{%
  MR\addcolon\space\href{http://www.ams.org/mathscinet-getitem?mr=MR#1}{#1}
}
\DeclareFieldFormat{zbl}{%
  Zbl\addcolon\space\href{http://zbmath.org/?q=an:#1}{#1}
}
\renewbibmacro*{eprint}{%
  \printfield{arxiv}%
  \newunit\newblock
  \printfield{mr}%
  \newunit\newblock
  \printfield{zbl}%
  \newunit\newblock
  \iffieldundef{eprinttype}
  {\printfield{eprint}}
  {\printfield[eprint:\strfield{eprinttype}]{eprint}}
}

\AtEveryBibitem{%
  \clearlist{address}%
}
\DeclareFieldFormat[article,inproceedings,inbook,incollection,thesis]{title}{\textit{#1}}
\renewbibmacro{in:}{}
\newcommand{\printreferences}{\raggedright\printbibliography[heading=bibintoc]}
\usepackage[inline,shortlabels]{enumitem}

\usepackage{subcaption}

\usepackage[yyyymmdd]{datetime}

\usepackage{etoolbox}
\ifundef{\abstract}{}{\patchcmd{\abstract}%
    {\quotation}{\quotation\noindent\ignorespaces}{}{}}

\usepackage[super]{nth}

\usepackage{thmtools}

\numberwithin{equation}{section}
\renewcommand{\qedsymbol}{$\blacksquare$}

\newcommand{\CorollaryQED}{\qedsymbol}

\newcommand{\ConjectureQED}{$\square$}
\newcommand{\SituationQED}{$\times$}
\newcommand{\DefinitionQED}{$\spadesuit$}
\newcommand{\NotationQED}{$\blacktriangleright$}
\newcommand{\ExampleQED}{$\bullet$}
\newcommand{\RemarkQED}{$\clubsuit$}

\declaretheorem[numberlike=equation,]{theorem}
\declaretheorem[numbered=no,name=Theorem]{theorem*}
\declaretheorem[numberlike=equation,name=Lemma]{lemma}
\declaretheorem[numberlike=equation,name=Proposition]{prop}
\declaretheorem[numberlike=equation,name=Corollary,qed=\CorollaryQED]{cor}

\declaretheorem[numberlike=equation,name=Hypothesis]{hypothesis}

\declaretheorem[numberlike=equation,name=Definition,style=definition,qed=\DefinitionQED]{definition}
\declaretheorem[numbered=no,name=Definition,style=definition,qed=\DefinitionQED]{definition*}

\declaretheorem[numberlike=equation,style=definition,qed=\ExampleQED]{example}

\declaretheorem[numberlike=equation,style=remark,qed=\RemarkQED]{remark}
\declaretheorem[numbered=no,style=remark,name=Remark,qed=\RemarkQED]{remark*}

\def\makeautorefname#1#2{\AtBeginDocument{\expandafter\def\csname#1autorefname\endcsname{#2}}}
\makeautorefname{table}{Table}        
\makeautorefname{chapter}{Chapter}
\makeautorefname{section}{Section}
\makeautorefname{subsection}{Section}
\makeautorefname{subsubsection}{Section}
\makeautorefname{footnote}{Footnote}
\AtBeginDocument{\def\itemautorefname~#1\null{(#1)\null}}
\AtBeginDocument{\def\equationautorefname~#1\null{(#1)\null}}

\numberwithin{substep}{step}
\makeautorefname{step}{Step}
\makeautorefname{substep}{Step}

\makeautorefname{case}{Case}
\makeautorefname{substep}{Step}

\setlist[description]{leftmargin=!,labelindent=1em}
\setlist[enumerate]{label={\rm (\arabic*)},ref=\arabic*}
\setlist[enumerate,2]{label={\rm (\alph*)},ref=\theenumi.\alph*}
\setlist[enumerate,3]{label={\rm (\roman*)},ref=\theenumii.\roman*}

\let\U\undefined

\usepackage{bm}
\usepackage{mathtools} 
\usepackage{stmaryrd} 

\DeclareFontFamily{U}{mathx}{\hyphenchar\font45}
\DeclareFontShape{U}{mathx}{m}{n}{
      <5> <6> <7> <8> <9> <10>
      <10.95> <12> <14.4> <17.28> <20.74> <24.88>
      mathx10
      }{}
\DeclareSymbolFont{mathx}{U}{mathx}{m}{n}
\DeclareFontSubstitution{U}{mathx}{m}{n}
\DeclareMathAccent{\widecheck}{0}{mathx}{"71}
\DeclareMathAccent{\wideparen}{0}{mathx}{"75}

\DeclareMathOperator{\Ad}{Ad}

\DeclareMathOperator{\End}{End}

\DeclareMathOperator{\GL}{GL}

\DeclareMathOperator{\HF}{\HF}

\DeclareMathOperator{\Lie}{Lie}

\DeclareMathOperator{\ad}{ad}

\DeclareMathOperator{\tr}{tr}

\DeclarePairedDelimiter{\norm}{\|}{\|}

\DeclarePairedDelimiterX{\inp}[2]{\langle}{\rangle}{#1, #2}

\DeclarePairedDelimiter{\abs}{\lvert}{\rvert}

\def\({\left(}
\def\){\right)}
\def\<{\left\langle}
\def\>{\right\rangle}

\newcommand{\PU}{{\P\U}}

\newcommand{\Ric}{\mathrm{Ric}}

\newcommand{\SO}{\mathrm{SO}}
\newcommand{\SU}{\mathrm{SU}}

\newcommand{\Spin}{\mathrm{Spin}}
\newcommand{\Sp}{\mathrm{Sp}}

\newcommand{\U}{\mathrm{U}}

\newcommand{\qandq}{\quad\text{and}\quad}

\renewcommand{\epsilon}{\varepsilon}

\newcommand{\su}{\mathfrak{su}}

\renewcommand{\Im}{\operatorname{Im}}

\renewcommand{\P}{\mathbf{P}}

\renewcommand{\setminus}{{\backslash}}

\renewcommand{\leq}{\leqslant}
\renewcommand{\geq}{\geqslant}

\newcommand{\w}{\wedge}
\newcommand{\tn}{\otimes}

\makeatletter
\renewcommand*\env@matrix[1][*\c@MaxMatrixCols c]{%
  \hskip -\arraycolsep
  \let\@ifnextchar\new@ifnextchar
  \array{#1}}

\renewcommand\xleftrightarrow[2][]{%
  \ext@arrow 9999{\longleftrightarrowfill@}{#1}{#2}}
\newcommand\longleftrightarrowfill@{%
  \arrowfill@\leftarrow\relbar\rightarrow}
\makeatother

\newcommand{\bA}{{\mathbf{A}}}

\newcommand{\bS}{{\mathbf{S}}}

\newcommand{\sA}{\mathscr{A}}

\newcommand{\sE}{\mathscr{E}}

\newcommand{\sK}{\mathscr{K}}

\newcommand{\fg}{{\mathfrak g}}

\newcommand{\fs}{{\mathfrak s}}

\newcommand{\fu}{{\mathfrak u}}

\newcommand{\fR}{{\mathfrak R}}

\newcommand{\slD}{\slashed D}

\author{Gorapada Bera}
\title{Growth of spinors in the generalized Seiberg--Witten equations on $\mathbb R^4$ and $\mathbb R^3$}
\date{\vspace{-5ex}}

\begin{document}
\maketitle
\begin{abstract}The classical Seiberg--Witten equations in dimensions three and four admit a natural generalization within a unified framework known as the generalized Seiberg--Witten (GSW) equations, which encompasses many important equations in gauge theory. This article proves that the averaged $L^2$-norm of any spinor with non-constant pointwise norm in the GSW equations on $\mathbb R^4$ and $\mathbb R^3$, measured over large-radius spheres, grows faster than a power of the radius, under a suitable curvature decay assumption. Separately, it is shown that if the Yang--Mills--Higgs energy of any solution of these equations is finite, then the pointwise norm of the spinor in it must converge to a non-negative constant at infinity. These two behaviors cannot occur simultaneously unless the spinor has constant pointwise norm. This work may be seen as partial generalization of results obtained by \citet{Taubes2017a}, and \citet{Nagy2019} for the Kapustin--Witten equations. 
\end{abstract}
\section{Introduction}The classical Seiberg--Witten (SW) equations \cites{Seiberg1994} can be generalized to a framework that contains many important gauge theoretic equations \cites{Taubes1999b, Pidstrygach2004, Haydys2008, Nakajima2015}. This framework requires a quaternionic representation $\rho:H\to \Sp(S)$ of a compact Lie group $H$ and a $\Spin^H$-structure (an extension of a $\Spin$ or $\Spin^c$-structure) on a smooth oriented Riemannian $4$-manifold $X$. Then the generalized Seiberg--Witten (GSW) equations are formulated as follows: for a connection $A$ inducing a fixed auxiliary connection $B$ and a spinor $\Phi$,
\begin{equation}\label{eq GSW 4 d formal}
	\begin{split}
	\slD_A\Phi=0,\\
	 F^+_{\ad(A)}=\mu(\Phi),
	 \end{split}
\end{equation}
where $\slD_A$ is the Dirac operator, and $\mu:\bS\to \Lambda^+(T^*X)\tn \ad(\fs)$ is a distinguished hyperkähler moment map.  For further details, see \autoref{sec GSW 4d}. Here $\ad(A)$ refers to the induced connection of $A$, induced by the adjoint representation of a compact Lie subgroup $G\subseteq H$, known as the \textit{structure group}. This unifying framework includes the anti-self duality (ASD) equations \cite{Donaldson1990}, the classical Seiberg--Witten equations \cite{Seiberg1994}, the $\U(n)$-monopole equations \cite{Feehan1998}, the Seiberg--Witten equations with multiple spinors \cite{Bryan1996}, the Vafa--Witten equations \cite{Vafa1994}, the complex ASD equations \cites{Taubes2013} which is closely related to the Kapustin--Witten equations \cite{Kapustin2007}, and the $\rm{ADHM}_{r,k}$ Seiberg--Witten equations \cite{Walpuski2019}. These equations not only play a pivotal role in physics, but are also likely to play an important role to the definition of invariants in low-dimensional topology \cites{Donaldson1990,Morgan1996,Witten2012}, as well as in higher-dimensional manifolds with special holonomy \cite{Doan2017d,Haydys2017}.

Focusing on the Euclidean space $X=\mathbb R^4$,  it is natural to ask questions about the solution space of the equations \autoref{eq GSW 4 d formal}, particularly about their behavior at infinity.  In this context, a fundamental question emerges: Do there exist any non-trivial solutions $(A,\Phi)$ to the equations \autoref{eq GSW 4 d formal} with finite Yang--Mills--Higgs (YMH) energy $\sE_4(A,\Phi)$?  The YMH energy functional is given by
\[\sE_4(A,\Phi)=\int_X \frac 12 \abs{F_{\ad(A)}}^2+\abs{\nabla_A\Phi}^2 +\abs{\mu(\Phi)}^2+\inp{\fR^+ \Phi}{\Phi},\]
where $\fR^+$ is the auxiliary curvature operator (see \autoref{def auxiliary curvature 4d}). It is worth noting that in most examples of generalized Seiberg--Witten (GSW) equations $\fR^+$ vanishes. The questions have been addressed for Kapustin--Witten equations with structure group $G=\SU(2)$ by \citet{Taubes2017a}, and \citet{Nagy2019}. Motivated by their work, we prove in the following theorem that the averaged $L^2$-norm of any spinor in the equations \autoref{eq GSW 4 d formal} with non-constant pointwise norm over large-radius spheres grows faster than a power of the radius, under a suitable curvature decay assumption. We also prove that, if the Yang--Mills--Higgs energy of any solution of these equations is finite, then the pointwise norm of the spinor in it must converge to a non-negative constant at infinity.
\begin{theorem}\label{growththeorem4}
Suppose $X = \mathbb R^4$ is equipped with the standard Euclidean metric and orientation, and the auxiliary connection $B$ is chosen so that the auxiliary curvature operator $\fR^+ = \tilde{\gamma}(F_B^+) \in \End(\bS^+)$ (see \autoref{def auxiliary curvature 4d}) vanishes. Let $(A, \Phi)$ be a solution to the generalized Seiberg--Witten equations \autoref{eq GSW 4 d formal}, or more generally, to the Euler--Lagrange equations \autoref{eq:EL4} associated with the Yang--Mills--Higgs energy functional $\sE_4$. Denote by $r$ the radial distance function from the origin in $\mathbb R^4$. 
\begin{enumerate}
    \item 
    \label{growththeorem4_part1}
    If $(A, \Phi)$ solves the equations \autoref{eq GSW 4 d formal}, assume that the anti-self-dual curvature, 
    $$ F^-_{\ad(A)} = o(r^{-2}) \quad \text{as } r \to \infty;$$  whereas if it solves the equations \autoref{eq:EL4}, assume instead that the curvature, 
    $$F_{\ad(A)} = o(r^{-2})  \quad \text{as } r \to \infty.$$
    Then either
    \[
    \nabla_A \Phi = 0 \qandq \mu(\Phi) = 0 \quad \text{(i.e., $|\Phi|$ is constant),}
    \]
    or there exists a constant $\varepsilon > 0$ such that
    \[
    \liminf_{r \to \infty} \frac{1}{r^{3+\varepsilon}} \int_{\partial B_r} |\Phi|^2 > 0.
    \]
    
    \item 
    \label{growththeorem4_part2}
    If $\sE_4(A, \Phi) < \infty$, then there exists a constant $m \geq 0$ such that
    \[
    |\Phi| - m = o(1) \quad \text{as } r \to \infty.
    \]
\end{enumerate}
\end{theorem}
By combining the contrasting behaviors established in \autoref{growththeorem4_part1} and \autoref{growththeorem4_part2} of \autoref{growththeorem4}, we obtain the following corollary, which asserts that the spinor must be parallel and the moment map vanishes. 
\begin{cor}\label{cor_vanishing4d}
Let $(A, \Phi)$ be as in \autoref{growththeorem4}, satisfying both assumptions in \autoref{growththeorem4_part1} and \autoref{growththeorem4_part2}. Then
\[
\nabla_A \Phi = 0 \qandq \mu(\Phi) = 0 \quad \text{(i.e., $|\Phi|$ is constant).} \qedhere
\]
\end{cor}

\begin{remark} 
\autoref{cor_vanishing4d} can be proved with the finite YMH energy assumption in \autoref{growththeorem4_part2} of \autoref{growththeorem4} alone, by adapting the arguments presented in \cite[Proposition 2.1]{Jaffe1980} with the divergence free symmetric $(0,2)$ tensor $T$ defined in \autoref{def T4d}; see also \autoref{lem div T is 0}.
\end{remark}

\begin{remark}
The idea behind the proof of \autoref{growththeorem4}~\autoref{growththeorem4_part1} traces back to establishing the monotonicity of a suitable (modified) frequency function--an approach originally employed by \citet{Taubes2017a} in the setting of the Kapustin--Witten equations. Notably, Taubes' argument avoids assuming curvature decay by making clever use of a special property of the Lie algebra $\su(2)$ \cite[Equation 4.12]{Taubes2017a}. While this property does not hold for a general structure group $G$, assuming curvature decay offers an alternative route to reach the same conclusion in our setting.  With this assumption in place, the method not only generalizes naturally to any generalized Seiberg--Witten equations with arbitrary structure group, but also takes a streamlined approach that avoids the technically involved step in Taubes' proof of decomposing the spinor along every direction in $\mathbb R^4$ and analyzing a separate frequency function for each--which highlights a key novelty of this article. The proof of \autoref{growththeorem4}~\autoref{growththeorem4_part2} leverages Heinz trick ($\varepsilon$-regularity) applied to the Yang--Mills--Higgs energy density. This part of the proof is inspired by arguments of a similar nature found in  \cites{Nagy2019,Fadel2022}.
In this way, the present work also partially generalizes the results of \cite{Nagy2019} to any generalized Seiberg--Witten equations, another important aspect of this article.
\end{remark}
\begin{remark}
We expect that the results presented here can be extended to the setting where $X$ is an ALE or ALF gravitational instanton, since these spaces are asymptotic to $\mathbb R^4$ and $\mathbb R^3\times S^1$, possibly modulo a finite group action. As our focus lies on the behavior at infinity, the definitions of the averaged $L^2$-norm of the spinor over large-radius spheres and the associated frequency function still make sense in this context--by integrating over large balls whose boundaries are cross-sections of the ends.  We believe that the (almost) monotonicity and related properties should continue to hold, thereby allowing for a conclusion analogous to that of the present work. This would, in particular, provide partial generalizations of the results of \cites{Bleher23,Nagy2019} on the Kapustin--Witten equations with structure group $\SU(2)$.  
\end{remark}

We now shift our focus to three dimensions, where we anticipate obtaining similar results. The dimensional reduction of the four dimensional generalized Seiberg--Witten equations  \autoref{eq GSW 4 d formal} on $X=\mathbb R\times M$ reduces to the three dimensional generalized Seiberg--Witten Bogomolny equations on $M$. That is, for a connection $A$ inducing a fixed auxiliary connection $B$, a Higgs field $\xi$ and a spinor $\Phi$,
\begin{equation}\label{eq GSWb 3d formal}
\begin{split}
	\slD_A\Phi=-\rho(\xi)\Phi,\\
	 F_{\ad(A)}=*d_{\ad(A)}\xi+\mu(\Phi).
	 \end{split}
\end{equation}
 The Bogomolny monopole equations \cite{Hitchin1982}, extended Bogomolny monopole equations \cite{Witten2015}, Kapustin--Witten monopole equations \cite{Nagy2019}, Haydys monopole equations \cite{Nagy2020} are examples of the equations \autoref{eq GSWb 3d formal}. We again consider the Yang--Mills Higgs energy functional in  dimension three, 
\[\sE_3(A,\xi, \Phi)=\int_M \abs{F_{\ad(A)}}^2+\abs{\nabla_A\Phi}^2+\abs{\nabla_{\ad(A)}\xi}^2 +\abs{\rho(\xi)\Phi}^2+\abs{\mu(\Phi)}^2+\inp{\fR \Phi}{\Phi},\]
where $\fR$ is the auxiliary curvature operator (see \autoref{def auxiliary curvature 3d}).
Setting the Higgs field $\xi=0$ in the equations \autoref{eq GSWb 3d formal} yields the generalized Seiberg--Witten equations in dimension three:
\begin{equation}\label{eq GSW 3d formal}
\begin{split}
	\slD_A\Phi=0,\\
	 F_{\ad(A)}=\mu(\Phi).
	 \end{split}
\end{equation}

We will again focus on the Euclidean space 
$M=\mathbb R^3$ and prove the following theorem regarding solutions of the generalized Seiberg--Witten Bogomolny equations \autoref{eq GSWb 3d formal}, similar to \autoref{growththeorem4}. The only difference is that the curvature decay assumption now requires an additional condition on the decay of the covariant derivative of the Higgs field. However, if we know that the Higgs field is zero, i.e., the solution satisfies the generalized Seiberg--Witten equations \autoref{eq GSW 3d formal}, both of these assumptions are no longer necessary.

\begin{theorem}\label{growththeorem3}
Suppose $M = \mathbb R^3$ is equipped with the standard Euclidean metric and orientation, and the auxiliary connection $B$ is chosen such that the auxiliary curvature operator $\fR = \tilde{\gamma}(F_B) \in \End(\bS)$ (see \autoref{def auxiliary curvature 3d}) vanishes. Let $(A, \xi, \Phi)$ be a solution to the generalized Seiberg--Witten Bogomolny equations \autoref{eq GSWb 3d formal}, or more generally, to the Euler--Lagrange equations \autoref{EL eq} associated with the Yang--Mills--Higgs energy functional $\sE_3$.  Denote by $r$ the radial distance function from the origin in $\mathbb R^3$. 
\begin{enumerate}
    \item 
    \label{growththeorem3_part1}
Assume $$\nabla_{\ad(A)} \xi = o(r^{-3/2}) \qandq F_{\ad(A)} = o(r^{-3/2})  \quad \text{as } r \to \infty.$$ However, if $\xi=0$ and $(A, \Phi)$ solves  \autoref{eq GSW 3d formal},  these decay assumptions are not required.
Then either
    \[
    \nabla_A \Phi = 0, \quad \mu(\Phi) = 0 \qandq \rho(\xi)\Phi = 0 \quad \text{(i.e., $|\Phi|$ is constant),}
    \]
    or there exists a constant $\varepsilon > 0$ such that
    \[
    \liminf_{r \to \infty} \frac{1}{r^{2+\varepsilon}} \int_{\partial B_r} |\Phi|^2 > 0.
    \]

    \item 
    \label{growththeorem3_part2}
    If $\sE_3(A, \xi, \Phi) < \infty$, then there exist constants $m_1, m_2 \geq 0$ such that
    \[
    \abs{\xi} - m_1 = o(1) \qandq \abs{\Phi}- m_2 = o(1) \quad \text{as } r \to \infty.
    \]
\end{enumerate}
\end{theorem}
Combining the contrasting behaviors from \autoref{growththeorem3_part1} and \autoref{growththeorem3_part2} of \autoref{growththeorem3}, we deduce the following corollary:  the spinor is parallel, and both the moment map and the action of the Higgs field vanish whenever both conditions are satisfied.
\begin{cor}
Let $(A, \xi, \Phi)$ be as in \autoref{growththeorem3}, satisfying both the assumptions in \autoref{growththeorem3_part1} and \autoref{growththeorem3_part2}. Then
\[
\nabla_A \Phi = 0, \quad \mu(\Phi) = 0 \qandq \rho(\xi)\Phi = 0 \quad \text{(i.e., $|\Phi|$ is constant).} \qedhere
\]
\end{cor}

\begin{remark} If $(A, \xi, \Phi)$ from \autoref{growththeorem3}  satisfies only the finite YMH energy assumption in \autoref{growththeorem3_part2} of \autoref{growththeorem3} then by following the arguments presented in \cite[Proposition 2.1]{Jaffe1980} with the divergence free symmetric $(0,2)$ tensor $T$ defined in \autoref{def T3d} (see also \autoref{lem3d div T is 0}), we would obtain the following equipartition identity, analogous to \cite[Corollary 2.2]{Jaffe1980}:
\[ \int_{\mathbb R^3} \abs{F_{\ad(A)}}^2=\int_{\mathbb R^3}\abs{\nabla_A\Phi}^2+\abs{\nabla_{\ad(A)}\xi}^2 +3\abs{\rho(\xi)\Phi}^2+3\abs{\mu(\Phi)}^2. \qedhere\]
\end{remark}

\paragraph{Acknowledgements.} I am deeply grateful to my PhD advisor, Thomas Walpuski, whose research on generalized Seiberg--Witten equations has greatly influenced this article. I also thank \'{A}kos Nagy and Gon\c{c}alo Oliveira for their work \cite{Nagy2019}, which has had a significant impact on the development of this paper. Additionally, I would like to sincerely thank the anonymous referee for several feedbacks and identifying an error in a lemma in one of the previous versions, whose corrected form led to the revisions in the main results of the current version.


\section{Generalized Seiberg--Witten equations in dimension four}\label{sec GSW 4d}
The primary objective of this section is to establish \autoref{growththeorem4}. To that end, we begin by laying the necessary groundwork on the generalized Seiberg--Witten equations in dimension four. This includes introducing the fundamental setup, clarifying the relevant notations, and deriving several key identities that will play a crucial role in the arguments to follow.
\subsection{Preliminaries: basic set up and identities}
The set up of generalized Seiberg--Witten equations in dimension four requires an algebraic and a geometric data which are  generalizations of datas we need to set up the classical Seiberg--Witten equations. Here we are closely following \cite{Walpuski2019,Walpuski2022}.  

\begin{definition}
  A \textbf{quaternionic hermitian vector space} is a left $\mathbb H$-module $S$ together with an inner product $\inp\cdot \cdot$ such that $i,j,k$ act by isometries. The \textbf{unitary symplectic group }$\Sp(S)$ is the subgroup of $\GL_{\mathbb H}(S)$ preserving $\inp\cdot \cdot$. 
  \end{definition}
 \begin{definition}An \textbf{algebraic data} is a triple $(H,\rho,G)$ where $H$ is a compact Lie group with $-1 \in Z(H)$ and $G$ is a closed, connected, normal subgroup of $H$, and $\rho:H\to \Sp(S)$ is a quaternionic representation of $H$. Here $S$ is a quaternionic hermitian vector space. The subgroup $G$ and the quotient group $K:=H/{\langle G,-1\rangle }$ are said to be the \textbf{structure group} and the \textbf{auxiliary group}, respectively.
 	\end{definition}
 
 Choose an algebraic data $(H,\rho,G)$. Denote the induced Lie algebra representation of $\rho{|_G}$ again by $\rho:\mathfrak g\to \End(S)$, where $\mathfrak g=\Lie(G)$. 
 Define  $\gamma:\mathbb H\to \End(S)$ and  $\tilde \gamma:$ $\Im \mathbb H \otimes \mathfrak g \to \End(S)$  by $$\gamma(v)\Phi:=v\cdot \Phi,\ \ \ \ \text{and}\ \ \ \tilde \gamma(v \otimes \xi):=\gamma(v)\circ \rho(\xi).$$ Then $\tilde \gamma^*:\End(S)\cong \End(S)^* \to $ ($\Im \mathbb H \otimes \mathfrak g)^*\cong (\Im \ \mathbb H)^*\otimes\mathfrak g$. Corresponding to the quaternionic representation $\rho{|_G}$ there is a distinguished \textbf{hyperk\"ahler moment map} $\mu:S\to(\Im \mathbb H)^* \otimes \mathfrak g$ defined by
	 $$\mu(\Phi):=\frac 12\tilde \gamma^*(\Phi \Phi^*),$$ 
	 that is, $\mu$ is $G$-equivariant and $\inp{(d\mu)_\Phi \phi}{v \otimes \xi}=\inp{\gamma(v)\rho(\xi)\Phi}{\phi}$ for all $v\in \Im \mathbb H$, $\xi\in \fg$ and $\Phi,\phi\in S$. Later we will identify $\Im \mathbb H$ with $\Lambda^+\mathbb H^*$ by the following isomorphism $v\mapsto \inp{dq\wedge d\bar q}{v}$, $q\in \mathbb H$.

Set $$\Spin^H(4):= \frac{\Sp(1)\times \Sp(1)\times H}{\{\pm 1\}}.$$ The group $\Sp(1)\times \Sp(1)$ acts on $\mathbb R^4\cong\mathbb H$ by $(p_+,p_-)\cdot x=p_-x \bar{p_+}$ and yields a $2$-fold covering $\Sp(1)\times \Sp(1)\to \SO(4)$ and therefore $\Spin(4)=\Sp(1)\times \Sp(1)$. 
 Define $\sigma_\pm:\Spin^H(4)\to \Sp(S)$ by
	$$ \sigma_\pm[p_+,p_-,z]=\gamma(p_\pm)\circ \rho(z).$$

\begin{definition}
	A \textbf{$\Spin^H$-structure} on an oriented Riemannian $4$-manifold $(X,g)$ is a principal $\Spin^H(4)$-bundle $\mathfrak s$ together with an isomorphism  \[\mathfrak s \times_{\Spin^H(4)} \SO(4)\cong \SO(TX).\qedhere\]
\end{definition}

Choose an algebraic data $(H,\rho,G)$. A $\Spin^H$-structure $\mathfrak s$ induces the following associated bundles and maps,
\begin{itemize}
	\item the \textbf{positive and negative spinor bundles}, $$\bS^\pm=\mathfrak s \times_{\sigma_\pm}S,$$
 \item the \textbf{adjoint bundle} and the \textbf{auxiliary bundle}, respectively, 
 \[\ad(\mathfrak s):=\mathfrak s \times_{\Spin^H(4)} \mathfrak g \qandq \sK:=\mathfrak s \times_{\Spin^H(4)} K,\]
	\item the \textbf{Clifford multiplication map} $\gamma: TX \to \End(\mathbf S^{+},\mathbf S^{-})$ induced by $\gamma$,
 \item $\tilde \gamma:TX\tn \ad(\fs)\to \End(S^+,S^-)$, induced by $\tilde \gamma$,
     \item the \textbf{moment map} $\mu:\mathbf S^+ \to \Lambda^+T^*X\otimes \ad(\fs),$ defined by
	 $$\mu(\Phi):=\frac 12\tilde \gamma^*(\Phi \Phi^*).$$ 
\end{itemize}

\begin{definition}
	A \textbf{geometric data} is a tuple $(X,g,\fs,B)$ where
	$\mathfrak s$ is a $\Spin^H$-structure  on an oriented Riemannian $4$-manifold $(X,g)$ and $B$ is a connection on the auxiliary bundle $\sK$.
\end{definition}

Choose a geometric data $(X,g,\fs,B)$. Denote by $\sA(\fs,B)$ the space of all connections on $\fs$ inducing the Levi-Civita connection on $TX$ and the connection $B$ on the auxiliary bundle $\sK$. For $A\in \sA(\fs,B)$ we denote the induced connection on $\ad(\fs)$ by $\ad(A)$. Note that $\sA(\fs,B)$ is nonempty and is an affine space over $\Omega^1(X,\ad(\fs))$. Every $A\in \sA(\fs,B)$ defines a \textbf{Dirac operator} 
$\slD_A: \Gamma(\mathbf S^+)\to \Gamma(\mathbf S^-)$ which is given by 
$$ \slD_A\Phi=\sum_{i=1}^{4} \gamma(e_i)\nabla_{A,{e_i}}\Phi,$$
	where $ \{ e_1,e_2,e_3,e_4\}$ is an oriented local orthonormal frame of $TX$.

\begin{definition}
	The \textbf{generalized Seiberg--Witten (GSW) equations in dimension four} associated with the datas $(H,\rho,G)$ and $(X,g,\fs,B)$ are the following equations for $A\in\sA(\fs,B) $, $\Phi \in  \Gamma(\mathbf S^+)$:
	\begin{equation} \label{sw eq}
	\slD_A\Phi=0,\ \ \ 
	F^+_{\ad(A)}=\mu(\Phi).  
	\end{equation}
 Solutions of the equations \autoref{sw eq} are called \textbf{generalized Seiberg--Witten (GSW) monopoles}.
\end{definition}

\begin{definition}\label{def auxiliary curvature 4d} We define the \textbf{auxiliary curvature operator} $\fR^+\in \End(\bS^+)$ by 
\[\fR^+:=\frac{\text{scal}_g}{4} + \tilde\gamma(F^+_B).\qedhere\]
\end{definition}
\begin{example}[{\textbf{ASD instantons}}]If $H=G\times \{\pm 1\}$ and $S=0$ then the GSW equations \autoref{sw eq} reduces to the anti-self duality (ASD) equations \cite{Donaldson1990} for a principal $G$-bundle. In this case obviously $\fR^+=0$.
\end{example}
\begin{example}[{\textbf{Harmonic spinors}}]If $H=\{\pm 1\}$ and $G=\{1\}$  then the GSW equations \autoref{sw eq} reduces to a Dirac equation whose solutions are harmonic spinors. In this case $\fR^+=\frac{\text{scal}_g}{4}$.
\end{example}
\begin{example}[{\textbf{Seiberg--Witten equations}}]If $H=G=\U(1)$, $S=\mathbb H$ and $\rho:\U(1)\to \Sp(1)$ is given by $$z\cdot q=qz\in \mathbb H=\mathbb C\oplus j\mathbb C$$ then the GSW equations \autoref{sw eq} reduces to the classical Seiberg--Witten equations (for more details see \cite[Example 1.1]{Walpuski2019}). In this case $\fR^+=\frac{\text{scal}_g}{4}$.
\end{example}
\begin{example}[{\textbf{$\Sp(1)$-Seiberg--Witten equations}}]If $H=G=\Sp(1)$, $S=\mathbb H$ and $\rho:\Sp(1)\to \Sp(1)$ is given by $$\rho(p)q=q\bar p$$ then the GSW equations \autoref{sw eq} reduces to the $\Sp(1)$-Seiberg--Witten equations (see \cite{Okonek1996}). In this case $\fR^+=\frac{\text{scal}_g}{4}$.
\end{example}
\begin{example}[{\textbf{$\U(n)$-monopole equations}}]If $H=G=\U(n)$, $S=\mathbb H\tn_{\mathbb C}\mathbb C^n$ and $\rho:\U(n)\to \Sp(S)$ is given by $$\rho(A) (q\tn w)=q\tn Aw$$ then the GSW equations \autoref{sw eq} reduces to the $\U(n)$-monopole equations (closely related to the $\PU(2)$-monopole equations studied in  \cite{Feehan1998}). In this case $\fR^+=\frac{\text{scal}_g}{4}$.
\end{example}
\begin{example}[{\textbf{Seiberg--Witten equations with $n$ spinors}}]If $H=G=\U(1)$  and $S=\mathbb H^n$ and $\rho:\U(1)\to \Sp(S)$ is given by $$\rho(z) (q_1,\dots,q_n)=(q_1z,\dots,q_nz)$$ then the GSW equations \autoref{sw eq} reduces to the Seiberg--Witten equations with $n$ spinors (see \cite{Bryan1996}). In this case $\fR^+=\frac{\text{scal}_g}{4}$.
\end{example}
\begin{example}[{\textbf{Vafa--Witten equations}}]\label{eg VW}Suppose $H=\Sp(1)\times G$ and $S=\mathbb H\tn_\mathbb R{\fg}$ and $\rho:\Sp(1)\times G\to \Sp(S)$ is given by $$\rho(p,g) (q\tn \xi)=q\bar p\tn \Ad(g)\xi.$$ The embedding ${\Sp(1)\times \Sp(1)}/{\{\pm 1\}}\hookrightarrow \Spin^{\Sp(1)}(4)$ given by $[p,q]\mapsto [p,q,p]$ and a principal $G$-bundle $P$, induce a $\Spin^H(4)$-structure on $X$. $B$ is induced by the Levi-Civita connection. Then the GSW equations \autoref{sw eq} reduces to the Vafa--Witten equations (see \cites{Mares2010,Taubes2017}). In this case,
 \[\bS^+=(\underline{\mathbb R}\oplus \Lambda^+T^*X)\tn \ad(P) \qandq \bS^-=T^*X\tn \ad(P),\] 
 and $\fR^+$ is a combination of scalar curvature and self-dual Weyl curvature.
\end{example}
\begin{example}[{\textbf{Complex ASD instanton}}]Suppose $H,G,S,\rho$ as in \autoref{eg VW}.  The embedding ${\Sp(1)\times \Sp(1)}/{\{\pm 1\}}\hookrightarrow \Spin^{\Sp(1)}(4)$ given by $[p,q]\mapsto [p,q,q]$ and a principal $G$-bundle $P$, induce a $\Spin^H(4)$-structure on $X$. $B$ is induced by the Levi-Civita connection. Then the GSW equations \autoref{sw eq} reduces to the complex ASD equations (see \cites{Taubes2013}). In this case, 
\[\bS^-=(\underline{\mathbb R}\oplus \Lambda^-T^*X)\tn \ad(P) \qandq \bS^+=T^*X\tn \ad(P),\] 
and $\fR^+=\Ric_g$.
\end{example}
\begin{example}[{\textbf{ADHM$_{r,k}$-Seiberg--Witten equations}}]If $H=\SU(r)\times \Sp(1)\times \U(k)$, $G=\U(k)$  and $S=\mathbb Hom_\mathbb C(\mathbb C^r,\mathbb H\tn_\mathbb C\mathbb C^k)\oplus \mathbb H\tn_\mathbb R\fu(k)$ and $\rho:H\to \Sp(S)$ is induced from the previous three examples, then the GSW equations \autoref{sw eq} reduces to the ADHM$_{r,k}$-Seiberg--Witten equations (see \cite[Example 1.15]{Walpuski2019}). 
\end{example}
\begin{prop}[{Lichenerowicz--Weitzenb\"ock formula, \cite[Proposition 5.1.5]{Morgan1996}}]\label{prop bochner weitzen}
	Suppose $A\in\sA(\fs,B)$ and $\Phi \in\Gamma(\mathbf S^+)$. Then
	\begin{equation*}
	\slD_A^*\slD_A\Phi=\nabla^*_A\nabla_A\Phi+\tilde\gamma(F^+_{\ad(A)})\Phi+ \fR^+\Phi.
	\end{equation*}
\end{prop}

The following identities, whose proofs are similar to the proofs of the identities in \cite[Appendix B]{Doan2017a} for dimension three, will be useful in later sections. 
\begin{prop}\label{prop identitiy}
	For $\Phi \in\Gamma(\mathbf S^+)$, we have
	$\inp{\tilde\gamma(\mu(\Phi))\Phi}{\Phi}=2\abs {\mu(\Phi) }^2$
	and 
	$$d_{\ad(A)}^*\mu(\Phi)=2*\mu(\slD_A\Phi,\Phi)-\rho^*((\nabla_A\Phi)\Phi^*)$$
\end{prop}

We define a Yang--Mills--Higgs energy (YMH) functional on the space $\sA(\fs, B)\times \Gamma(\bS^+)$ which maps $(A,\Phi)\to \sE_4(A,\Phi)\in \mathbb R$. We will also see in the following that on an oriented closed $4$-manifold absolute minima of this functional are generalized Seiberg--Witten monopoles.

\begin{definition}
We define the \textbf{Yang--Mills--Higgs energy functional} $\sE_4:\sA(\fs,B) \times  \Gamma(\mathbf S^+)\to \mathbb R$ by
	\[\sE_4(A,\Phi)=\int_X \frac 12 \abs{F_{\ad(A)}}^2+\abs{\nabla_A\Phi}^2 +\abs{\mu(\Phi)}^2+\inp{\fR^+ \Phi}{\Phi}.\qedhere\]
\end{definition}
\begin{remark}\label{rmk energy equal}
If $X$ is closed then for $A\in\sA(\fs,B) $ and $\Phi \in  \Gamma(\mathbf S^+)$ we obtain using \autoref{prop bochner weitzen} that
	$$\sE_4(A,\Phi)=\int_X \abs{F^+_{\ad(A)}-\mu(\Phi)}^2+\abs{\slD_A\Phi}^2+8\pi^2 \Check{h}(G)k(\ad(A)),$$
	where $k(\ad(A)):=\frac 1{8\pi^2\Check{h}(G)}\int_X \langle F_{\ad(A)}\wedge F_{\ad(A)}\rangle$ is a constant topological term, called instanton number and $\Check{h}(G)$ is the dual Coxeter number of $G$. Indeed,
 \begin{align*}
	&\int_X \abs{F^+_{\ad(A)}-\mu(\Phi)}^2+\abs{\slD_A\Phi}^2 \\
	&=\int_X \abs{F^+_{\ad(A)}}^2+\abs{\mu(\Phi)}^2-2\inp{F^+_{\ad(A)}}{\mu(\Phi)}+\inp{\slD_A^*\slD_A\Phi}{\Phi}\\
	&=\int_X \abs{F^+_{\ad(A)}}^2+\abs{\mu(\Phi)}^2+\abs{\nabla_A\Phi}^2+ \inp{\fR^+ \Phi}{\Phi}=\sE_4(A,\Phi)-\int_X \langle F_{\ad(A)}\wedge F_{\ad(A)}\rangle.
	\end{align*}
 Therefore the absolute minima of this functional are generalized Seiberg--Witten monopoles.
\end{remark}

\begin{prop}
	The \textbf{Euler--Lagrange equations} for the Yang--Mills--Higgs energy functional $\sE_4$ are the following equations: for $A\in\sA(\fs,B) $, $\Phi \in  \Gamma(\mathbf S^+)$,
	\begin{equation}\label{eq:EL4}
	\begin{split}
	d_{\ad(A)}^*F_{\ad(A)}&=-2\rho^*((\nabla_A\Phi)\Phi^*),\\
	\nabla_A^*\nabla_A\Phi&=-\tilde \gamma(\mu(\Phi))\Phi -\fR^+\Phi
	\end{split}
	\end{equation}
\end{prop}
\begin{proof}Suppose $A \in\sA(\fs,B)$, $a\in \Omega^1(X,\ad(\fs))$, $\Phi,\phi \in  \Gamma(\mathbf S^+)$. Assume that $a, \phi$ are compactly supported. The proof requires only the following direct computations. For $\abs{t}\ll 1$ we obtain
$$\frac 12\frac d {dt}\|F_{\ad(A)+ta}\|_{L^2}^2=\inp{d_{\ad(A)}^*F_{\ad(A)}}{a}_{L^2}+O(t),$$
$$\frac d {dt}\|\nabla_{A+ta}(\Phi+t\phi)\|_{L^2}^2=2\inp{\nabla_A^*\nabla_A\Phi}{\phi}_{L^2}+2\inp{\rho^*((\nabla_A\Phi)\Phi^*)}{a}_{L^2}+O(t),$$
$$\frac d {dt}\|\mu(\Phi+t\Psi)\|_{L^2}^2=2\inp{\tilde \gamma(\mu(\Phi))\Phi}{\phi}_{L^2}+O(t),$$
and
\[\frac d {dt}\inp{\fR^+ (\Phi+t\phi)}{\Phi+t\phi}_{L^2}=2\inp{\fR^+\Phi}{\phi}_{L^2}+O(t). \qedhere\]
	\end{proof}
\begin{remark}
	If $(A,\Phi)$ is a GSW monopole then it satisfies the Euler--Lagrange equations \autoref{eq:EL4}. Indeed, this follows from \autoref{rmk energy equal} directly. Alternatively we can do a direct computation with the help of \autoref{prop identitiy}:
	\begin{align*}
		d^*_{\ad(A)} F_{\ad(A)}=2 d^*_{\ad(A)} F_{\ad(A)}^+&=2d^*_{\ad(A)}\mu(\Phi)\\
		&= 4*\mu(\slD_A\Phi,\Phi)-2\rho^*((\nabla_A\Phi)\Phi^*)= -2\rho^*((\nabla_A\Phi)\Phi^*).
		\end{align*}
Lichenerowicz--Weitzenb\"ock formula of \autoref{prop bochner weitzen} implies 
\[\nabla_A^*\nabla_A\Phi
		=-\tilde \gamma(\mu(\Phi))\Phi- \mathcal R\Phi. \qedhere\]
	\end{remark}
By taking inner product with $\Phi$ in the second equation of the equations \autoref{eq:EL4}, we derive the following Bochner identity as a corollary.
	\begin{cor}\label{Bochner1}
	Let $(A, \Phi)$ be a solution to the generalized Seiberg--Witten equations \autoref{eq GSW 4 d formal}, or more generally, to the Euler--Lagrange equations \autoref{eq:EL4}. Then
	\begin{equation}\label{eq bochner 4 cor}
	\frac 12\Delta\abs{\Phi}^2+\abs{\nabla_A\Phi}^2+2\abs {\mu(\Phi) }^2 +\inp{\fR^+\Phi}{\Phi}=0. \qedhere
	\end{equation}
\end{cor}
The following corollary is obtained by applying an integration by parts to the above Bochner identity.
\begin{cor}
	Let $\Omega$ be a bounded open subset of $X$ with smooth boundary $\partial \Omega$  and $f\in C^\infty(\bar \Omega)$. Suppose $(A,\Phi)$ satisfies the equations \autoref{eq bochner 4 cor} on $\Omega$, then 
	\[\frac 12 \int_\Omega \Delta f \cdot \abs{\Phi}^2+\int_\Omega f \cdot (\abs{\nabla_A\Phi}^2+2\abs {\mu(\Phi) }^2)=-\int_\Omega f \cdot \inp{\fR^+\Phi}{\Phi}+\frac 12 \int_{\partial \Omega} f\cdot \partial_\nu\abs{\Phi}^2-\partial_\nu f \cdot\abs{\Phi}^2.\qedhere\]
\end{cor}
The next proposition highlights how the maximum principle imposes significant restrictions on the behavior of GSW monopoles under the assumption of non-negative self-dual auxiliary curvature. 
\begin{prop}
	Let $(X,g)$ be an oriented Riemannian $4$-manifold and $(A,\Phi)$ be a GSW monopole or more generally a solution of the Euler--Lagrange equations \autoref{eq:EL4}. Assume $\fR^+\geq0 $ (\text{i.e.} $\inp{\fR^+\Phi}{\Phi}\geq 0 \  \forall \Phi \in \Gamma(\bS^+))$. 
	\begin{enumerate}[(i)] 
		\item If $X$ is closed then $\abs{\Phi}$ is constant, or equivalently \[\nabla_A\Phi= 0, \quad \mu(\Phi)=0 \qandq \inp{\fR^+\Phi}{\Phi}=0.\]
	\item If $X$ is noncompact and $\abs{\Phi}^2$ decays to zero at infinity then $\Phi=0$.\qedhere
	\end{enumerate}
\end{prop}
\begin{proof}
	Since $\frac 12 \Delta \abs{\Phi}^2= \inp{\nabla_A^*\nabla_A\Phi}{\Phi}-\abs{\nabla_A\Phi}^2=-2\abs{\mu(\Phi)}^2-\abs{\nabla_A\Phi}^2-\inp{\fR^+\Phi}{\Phi}\leq 0$,  $\abs{\Phi}^2$ is subharmonic. This implies the required assertions after applying the maximum principle.
\end{proof}

\subsection{Frequency function and the proof of \autoref{growththeorem4} (\ref{growththeorem4_part1})} \label{frequency_4d}
Throughout this subsection, we impose the following standing assumption, which is a part of \autoref{growththeorem4}.
\begin{hypothesis}\label{hyp main 4d} $X=\mathbb R^4$ with the standard Euclidean metric and orientation, and the auxiliary connection $B$ is chosen so that the auxiliary curvature operator $\fR^+=\tilde \gamma(F_B^+)\in \End(\bS^+)$ vanishes.
\end{hypothesis}

Let $(A, \Phi)$ be a solution to the generalized Seiberg--Witten equations \autoref{eq GSW 4 d formal}, or more generally, to the Euler--Lagrange equations \autoref{eq:EL4} associated with the Yang--Mills--Higgs energy functional $\sE_4$. Denote by $r$ the radial distance function from the origin in $\mathbb R^4$. \autoref{growththeorem4}~\autoref{growththeorem4_part1} concerns the asymptotic behavior of the $L^2$-norm of $\Phi$ averaged over spheres of radius $r$ as $r \to \infty$. To investigate this behavior, we employ Almgren's frequency function $N(r)$, originally introduced in the context of harmonic functions \cite{Almgren1979} and later adapted to gauge theory by \citet{Taubes2012}, along with a slight modification suited to our setting. The strength of this (modified) frequency function lies in its monotonicity, which controls growth behavior of the spinor. In particular, a uniform lower bound $\alpha > 0$ on it implies that the averaged $L^2$-norm of $\Phi$ at least grow like $r^{\alpha}$.  Our treatment closely follows the approach in \cite{Walpuski2019}.
\begin{definition} Denote by $B_{r}$ the open ball in $\mathbb R^4$ centered at $0$. 
	\begin{enumerate}[i)]
		\item For every $r>0$ we define 
		$$m(r):=\frac 1{r^3}\int_{\partial B_r}\abs{\Phi}^2\ \ \  \text{and}\ \ \ D(r):=\frac 1{r^2}\int_{B_r}\abs{\nabla_A \Phi}^2+2\abs{\mu(\Phi)}^2.$$
		\item Set $r_{-1}:= \sup \{0,r: r\in (0,\infty):m(r)=0\}$. The \textbf{frequency function} $N: (r_{-1},\infty)\to [0,\infty)$ is defined by 
		\[N(r):=\frac{D(r)}{m(r)}=\frac {r\int_{B_r}\abs{\nabla_A \Phi}^2+2\abs{\mu(\Phi)}^2 }{\int_{\partial B_r}\abs{\Phi}^2}.\qedhere\]
	\end{enumerate}
\end{definition}
Our objective is to analyze the monotonicity behavior of $N(r)$ , and for that, we need to compute its derivative, $N^\prime(r)$. To begin, we first calculate the derivative of the squared $L^2$-norm average of $\Phi$, $m(r)$ as follows:
\begin{prop}\label{prop m prime}For every $r>0$,
\[m^\prime(r)=\frac {2D(r)}r.\]
\end{prop}
\begin{proof}The proof is a direct computation.
	\begin{align*}
	m^\prime(r)&= \frac 1{ r^3} \frac {d}{dr}\int_{\partial B_r}\abs{\Phi}^2-\frac 3{r^4} \int_{\partial B_r}\abs{\Phi}^2\\
	&=\frac 1{r^3} (\int_{\partial B_r}\frac 3r\abs{\Phi}^2+ \int_{\partial B_r}\partial_r\abs{\Phi}^2)-\frac 3{r^4} \int_{\partial B_r}\abs{\Phi}^2=\frac 2{ r^3}\int_{B_r}\abs{\nabla_A \Phi}^2+2\abs{\mu(\Phi)}^2=\frac 2r D(r).\qedhere
	\end{align*}
\end{proof}
\begin{cor}\label{cor m prime}We have
\begin{enumerate}[a)]
		\item $m^\prime(r)\geq0, \forall r\in (0,\infty)$, and if $\Phi\neq 0$ then $r_{-1}=0$,
		\item for every $r\in (r_{-1},\infty)$, \[m^\prime(r)=\frac {2N(r)}r m(r).\qedhere\]
	\end{enumerate}
\end{cor}
Since $N(r)$ is the quotient of $D(r)$ and $m(r)$, we must also compute the derivative of $D(r)$. The following proposition provides the derivative:
\begin{prop}\label{propDprime4d}
For every $r>0$,
	\[D^\prime(r)=\frac 2{r^2}\int_{\partial B_r}\abs{\nabla_{A,\partial_r} \Phi}^2+\frac 12\abs{\iota(\partial_r)F_{\ad(A)}}^2+\frac 12 \abs{\mu(\Phi)}^2 -\frac 14\abs{F_{\ad(A)}}^2.\]
\end{prop}
To prove this proposition, we require a lemma about the divergence-free property of a certain symmetric $(0,2)$ tensor field $T$, similar to the approach by \citet[Proof of Lemma 5.2]{Taubes2012}.
\begin{definition}\label{def T4d}	
	The symmetric $(0,2)$ tensor $T$ is defined by $T:=T_1+T_{2}+T_3$ where
 \begin{align*}T_1(v,w)=\inp{\nabla_{A,v}\Phi}{\nabla_{A,w}\Phi}-\frac 12\inp{v}{w}\abs{\nabla_A \Phi}^2,\\
	2T_2(v,w)=\inp{\iota_vF_{\ad(A)}}{\iota_wF_{\ad(A)}}-\frac12 \inp{v}{w}\abs{F_{\ad(A)}}^2,\\
	T_3(v,w)=-\frac12 \inp{v}{w}\abs{\mu(\Phi)}^2. 
 \end{align*}
 Note that $\tr (T)=-\abs{\nabla_A \Phi}^2-2\abs{\mu( \Phi)}^2$.
\end{definition}
\begin{lemma}\label{lem div T is 0} The divergence of $T$ is given by:
  \[\nabla^*T=0.\]  
\end{lemma}
\begin{proof}
	Let $p\in \mathbb R^4$ and $\{e_i\}$ be an oriented orthonormal frame around $p$ such that $\nabla_{e_i}e_j(p)=0$. We have
	\begin{align*}
	(\nabla^*T_1)(e_i)&=-\sum_j\inp{\nabla_j \nabla_j\Phi}{\nabla_i\Phi}+\inp{\nabla_j\Phi}{\nabla_j \nabla_i\Phi}-\inp{\nabla_j\Phi}{\nabla_i \nabla_j\Phi}\\
	&= \inp{\nabla_A^*\nabla_A\Phi}{\nabla_i\Phi}+\sum_j\inp{\nabla_j\Phi}{F_{\ad(A)}(e_i,e_j)\Phi}\\
	&=-\inp{\tilde \gamma(\mu(\Phi))\Phi }{\nabla_i\Phi}+\sum_j\inp{\nabla_j\Phi}{\rho(F_{\ad(A)}(e_i,e_j))\Phi}\\
	&=-\inp{\mu(\Phi)}{\nabla_{\ad(A),e_i}\mu(\Phi)}+\sum_j\inp{\rho^*((\nabla_j\Phi)\Phi^*)}{F_{\ad(A)}(e_i,e_j)}\\
	&=-\frac 12\nabla_i\abs{\mu(\Phi)}^2+\inp{\rho^*((\nabla_A\Phi)\Phi^*)}{\iota_{e_i}F_{\ad(A)}},
	\end{align*}
	and
		\begin{align*}
	&2(\nabla^*T_2)(e_i)\\&=-\sum_j\inp{\nabla_j\iota_{e_i}F_{\ad(A)}}{\iota_{e_j}F_{\ad(A)}}+\inp{\nabla_j\iota_{e_j}F_{\ad(A)}}{\iota_{e_i}F_{\ad(A)}}+\frac 12 \nabla_i\abs{F_{\ad(A)}}^2\\
	&=-\sum_j\inp{{e_j}\w\iota_{e_i}\nabla_jF_{\ad(A)}}{F_{\ad(A)}}+\inp{\iota_{e_j}\nabla_jF_{\ad(A)}}{\iota_{e_i}F_{\ad(A)}}+\frac 12\nabla_i\abs{F_{\ad(A)}}^2\\
	&=\sum_j\inp{\iota_{e_i}{e_j}\w\nabla_jF_{\ad(A)}}{F_{\ad(A)}}-\inp{\nabla_iF_{\ad(A)}}{F_{\ad(A)}}+\inp{d_{\ad(A)}^*F_{\ad(A)}}{\iota_{e_i}F_{\ad(A)}}+\frac 12\nabla_i\abs{F_{\ad(A)}}^2\\
	&=\sum_j\inp{\iota_{e_i}d_{\ad(A)} F_{\ad(A)}}{F_{\ad(A)}}+\inp{d_{\ad(A)}^*F_{\ad(A)}}{\iota_{e_i}F_{\ad(A)}}=-2\inp{\rho^*((\nabla_A\Phi)\Phi^*)}{\iota_{e_i}F_{\ad(A)}}.
	\end{align*}
Since $\nabla^*T_3(e_i)=\frac 12\nabla_i\abs{\mu(\Phi)}^2$,  we obtain $\nabla^*T=0$.
\end{proof}
\begin{proof}[{Proof of \autoref{propDprime4d}}]
	We have $D^\prime(r) =-\frac 2{r} D(r)+ \frac 1{r^2}\int_{\partial B_r}\abs{\nabla_A \Phi}^2+2\abs{\mu(\Phi)}^2$. Now
	\begin{align*} 
	&0=\int_{B_r}\inp{\nabla^*T}{dr^2}\\
	&=-2r\int_{\partial B_r}T(\partial_r,\partial_r)+\int_{B_r}2\tr(T)\\
	&=-2r\int_{\partial B_r}\abs{\nabla_{A,\partial_r} \Phi}^2+\frac 12\abs{\iota(\partial_r)F_{\ad(A)}}^2-\frac 14\abs{F_{\ad(A)}}^2\\
	&+r\int_{\partial B_r}\abs{\nabla_A \Phi}^2+\abs{\mu(\Phi)}^2+\int_{B_r}2\tr(T)\\
	&=-2r\int_{\partial B_r}\abs{\nabla_{A,\partial_r} \Phi}^2+\frac 12\abs{\iota(\partial_r)F_{\ad(A)}}^2+\frac 12 \abs{\mu(\Phi)}^2-\frac 14\abs{F_{\ad(A)}}^2+r^3D^\prime(r).\qedhere
	\end{align*}
\end{proof}
We are now prepared to present the final formula for the derivative of the frequency function  $N(r)$ by combining the results from the two previous propositions.\begin{prop}\label{prop N prime }For all $r> r_{-1}$ we have
$$N^\prime(r)=\frac 2{r^2m(r)}\int_{\partial B_r}\abs{\nabla_{A,\partial_r} \Phi-\frac 1r N(r)\Phi}^2+\frac 12\abs{\iota(\partial_r)F_{\ad(A)}}^2+\frac 12 \abs{\mu(\Phi)}^2 -\frac 14\abs{F_{\ad(A)}}^2.$$
	\end{prop}
\begin{proof}
	Since	$D(r)=\frac 1{r^2}\int_{\partial B_r}\inp{\nabla_{A,\partial r}\Phi}{\Phi}$,  
	\begin{align*}
	N^\prime(r)&=\frac {D^\prime(r)}{m(r)}-D(r) \frac {m^\prime(r)}{m(r)^2}\\
	&=\frac 2{r^2m(r)}\int_{\partial B_r}\abs{\nabla_{A,\partial_r} \Phi}^2+\frac 12\abs{\iota(\partial_r)F_{\ad(A)}}^2+\frac 12 \abs{\mu(\Phi)}^2 -\frac 14\abs{F_{\ad(A)}}^2 -\frac 2r N(r)^2\\
	&=\frac 2{r^2m(r)}\int_{\partial B_r}\abs{\nabla_{A,\partial_r} \Phi-\frac 1r N(r)\Phi}^2+\frac 12\abs{\iota(\partial_r)F_{\ad(A)}}^2+\frac 12 \abs{\mu(\Phi)}^2 -\frac 14\abs{F_{\ad(A)}}^2.\qedhere
	\end{align*}
\end{proof}
\begin{remark}
\label{remark almostmonotonicity}
From the above proposition, it is evident that $N(r)$ may not exhibit monotonicity. However, if $(A, \Phi)$ is a solution to the generalized Seiberg--Witten equations \autoref{eq GSW 4 d formal}, then $ F^+_{\ad(A)}=\mu(\Phi)$. In this case, by using \autoref{prop N prime }, we obtain the inequality
\begin{equation}\label{eqnalmostmonotonesw}N^\prime(r)+ \frac 1{r^2m(r)}\int_{\partial B_r}\abs{F^-_{\ad(A)}}^2 \geq 0.\end{equation}
 On the other hand, if $(A, \Phi)$ satisfies only the Euler--Lagrange equations \autoref{eq:EL4} then we obtain instead the inequality
\begin{equation}\label{eqnalmostmonotoneEL}N^\prime(r)+ \frac 1{2r^2m(r)}\int_{\partial B_r}\abs{F_{\ad(A)}}^2 \geq 0.\end{equation}
As a result of this, we modify the frequency function $N(r)$ in the following proof, ensuring that it exhibits the necessary monotonicity, provided we are given the assumptions in  \autoref{growththeorem4}~\autoref{growththeorem4_part1}. 
\end{remark}
\begin{proof}[{{\normalfont{\textbf{Proof of \autoref{growththeorem4}~\autoref{growththeorem4_part1}}}}}]
	Assume $\Phi \neq0$. By \autoref{cor m prime}, $r_{-1}=0$. Evidently, $N(r)=0 \ \forall r>0$ if and only if $\nabla_A\Phi=0$ and $\mu(\Phi)=0$, or equivalently, by \autoref{Bochner1}, $\abs{\Phi}$ is constant.  
	
Therefore we can assume $N\neq0$. Given the assumptions in  \autoref{growththeorem4}~\autoref{growththeorem4_part1}, the inequalities \autoref{eqnalmostmonotonesw} and \autoref{eqnalmostmonotoneEL} in \autoref{remark almostmonotonicity}, ensure that for every $c>0$ there exists $\rho>0$ such that 
\begin{equation}\label{eq Nc}N^\prime(r)+\frac {2c}{r^{3}m(\rho)}\geq0, \forall r\geq\rho.\end{equation}
Define the modified frequency function: 
\[\widetilde N_c(r):=N(r)-\frac c{m(\rho)r^{2}}, \quad \forall r\geq\rho.\] 
It follows that ${\widetilde N_c(r)}^\prime\geq 0, \forall r\geq\rho$, which gives the desired almost monotonicity property. We claim that $\widetilde N_c$ controls $m$. To see this, observe that for all $r\geq\rho$,
\[m^\prime(r)=\Big(\frac {2\widetilde N_c(r)}r+\frac {2c}{m(\rho)r^{3}} \Big) m(r).\]
For $\rho\leq s< r<\infty$ and  $t\in[s,r]$ we have
	$$\frac {2\widetilde N_c(s)}t+\frac {2c}{m(\rho)t^{3}}\leq\frac{d}{dt}\log(m(t))\leq\frac {2\widetilde N_c(r)}t+\frac {2c}{m(\rho)t^{3}},$$
and therefore,
\begin{equation}\label{monotone1}
	\Big(\frac rs\Big)^{2\widetilde N_c(s)}\Big(e^{\int_s^r\frac {2c}{m(\rho)t^{3}} dt}\Big)m(s)\leq m(r)\leq \Big(\frac rs\Big)^{2\widetilde N_c(r)}\Big(e^{\int_s^r\frac {2c}{m(\rho)t^{3}} dt}\Big)m(s).
	\end{equation}	
From this estimate we can conclude the growth of $m(r)$ as follows:	

	\textbf{Case 1:}  There exists $c>0$ such that $\widetilde N_c(s)>0$ for some $s\geq\rho$. Set $\varepsilon:=2\widetilde N_c(s)>0$. The above estimate \autoref{monotone1} yields 
	$$m(r)\geq r^\varepsilon \frac{m(s)}{s^\varepsilon}\Big(e^{\int_s^r\frac {2c}{m(\rho)t^{3}} dt}\Big).$$
	Thus $$\liminf_{r\to\infty}\frac 1{r^{\varepsilon}} m(r)\gtrsim \frac{m(s)}{s^\varepsilon}>0.$$
		
 \textbf{Case 2:} There exist a decreasing sequence $\{c_n\}$ converging to $0$ and an increasing sequence $\{s_n\}$ of positive real numbers converging to $\infty$ as $n\to \infty$ such that $\widetilde N_{c_n}(s_n)\leq 0$. That is
 $$N(s_n)\leq\frac {c_n}{m(s_n)s_n^{2}} \  \text{and hence} \ s_n^2 D(s_n)\leq c_n \to0 \  \text{as} \ n\to\infty,$$which is a contradiction as $N\neq 0$. 
\end{proof}

\subsection{Consequence of finite energy and the proof of \autoref{growththeorem4} (\ref{growththeorem4_part2}) }\label{finite_energy4}In this section, we also assume \autoref{hyp main 4d}. Let $(A,\Phi)$ be a solution to the generalized Seiberg--Witten equations \autoref{eq GSW 4 d formal}, or more generally, to the Euler--Lagrange equations \autoref{eq:EL4} associated with the Yang--Mills--Higgs energy functional $\sE_4$. We will show that if $\sE_4(A,\Phi)$ is finite, then $\abs{\Phi}$ must converge to a non-negative constant $m$ at infinity. The key idea is to apply Heinz trick ($\varepsilon$-regularity) from \autoref{Heinz} to the energy density $e(A,\Phi)$, which is the integrand in the YMH energy functional $\sE_4$. The proof draws on several arguments of similar nature found in \cites{Nagy2019,Fadel2022}.
\begin{definition}
 The \textbf{energy density function} $e: \sA(\fs,B) \times \Gamma(\mathbf S^+)\to C^\infty(\mathbb R^4,\mathbb R)$ is defined by
	\[e(A,\Phi)=
	|F_{\ad(A)}|^2+|\nabla_A\Phi|^2+|\mu(\Phi)|^2.\qedhere\]
\end{definition}
\begin{lemma}\label{estimate1}
Let $(A, \Phi)$ be a solution to the generalized Seiberg--Witten equations \autoref{eq GSW 4 d formal}, or more generally, to the Euler--Lagrange equations \autoref{eq:EL4} associated with the Yang--Mills--Higgs energy functional $\sE_4$. Then
	$$\Delta e(A,\Phi) \lesssim e(A,\Phi)+e(A,\Phi)^{\frac 32}$$
\end{lemma}
\begin{proof}
	In the following computations, we are going to use either Lichenerowicz--Weitzenb\"ock formula for Lie-algebra bundle valued $2$-forms, or the Euler--Lagrange equations \autoref{eq:EL4}, or \autoref{prop identitiy}.
	
	\begin{align*}
	\frac 12\Delta \abs{F_{\ad(A)}}^2
	&\leq \inp{\nabla_{\ad(A)}^*\nabla_{\ad(A)} F_{\ad(A)}}{F_{\ad(A)}}\\
	&\lesssim \inp{\Delta_{\ad(A)} F_{\ad(A)}}{F_{\ad(A)}}+\abs{F_{\ad(A)}}^2+\abs{F_{\ad(A)}}^3\\
	&=\inp{-2d_{\ad(A)}\rho^*((\nabla_A\Phi)\Phi^*)}{F_{\ad(A)}}+\abs{F_{\ad(A)}}^2+\abs{F_{\ad(A)}}^3\\
	&=2\inp{-\rho^*((\rho(F_{\ad(A)})\Phi)\Phi^*)-\rho^*(\nabla_A\Phi \w(\nabla_A\Phi)^*)}{F_{\ad(A)}}+\abs{F_{\ad(A)}}^2+\abs{F_{\ad(A)}}^3 \\
	&\lesssim-2\abs{\rho(F_{\ad(A)})\Phi}^2+\abs{\nabla_A\Phi}^2\abs{F_{\ad(A)}}+\abs{F_{\ad(A)}}^2+\abs{F_{\ad(A)}}^3\\
	&\lesssim e(A,\Phi)+e(A,\Phi)^{\frac 32},
	\end{align*}
		\begin{align*}
	\frac 12\Delta|\nabla_A\Phi|^2
	&\leq \inp{\nabla_A^*\nabla_A\nabla_A\Phi}{\nabla_A\Phi}\\
	&=\inp{[\nabla_A^*\nabla_A,\nabla_A]\Phi}{\nabla_A\Phi}+\inp{\nabla_A\nabla_A^*\nabla_A\Phi}{\nabla_A\Phi} \\
	&\lesssim \inp{\rho(d_{\ad(A)}^*F_{\ad(A)})\Phi}{\nabla_A\Phi}+\abs{F_{\ad(A)}}\abs{\nabla_A\Phi}^2-\inp{\nabla_A(\tilde \gamma(\mu(\Phi))\Phi)}{\nabla_A\Phi} \\
	&\lesssim-2\abs{\rho^*(\nabla_A\Phi)\Phi^*}^2+\abs{F_{\ad(A)}}\abs{\nabla_A\Phi}^2 - 2\abs{\mu(\nabla_A\Phi,\Phi)}^2+\abs{\mu(\Phi)}\abs{\nabla_A\Phi}^2\\
 &\lesssim e(A,\Phi)+e(A,\Phi)^{\frac 32},
	\end{align*}	
		\begin{align*}
	\frac 12\Delta|\mu(\Phi)|^2
	\leq \inp{\nabla_{\ad(A)}^*\nabla_{\ad(A)}\mu(\Phi)}{\mu(\Phi)}
	&=2\inp{\mu(\nabla_A^*\nabla_A\Phi,\Phi)-\langle \mu(\nabla_A\Phi,\nabla_A\Phi)\rangle}{\mu(\Phi)}\\
	&=2\inp{-\mu(\tilde \gamma(\mu(\Phi))\Phi,\Phi)-\langle\mu(\nabla_A\Phi,\nabla_A\Phi)\rangle}{\mu(\Phi)}\\
	&\lesssim \abs{\mu(\Phi)}^3+\abs{\nabla_A\Phi}^2\abs{\mu(\Phi)}\lesssim e(A,\Phi)^{\frac 32}.
	\qedhere\end{align*}
	\end{proof}

\begin{proof}[{{\normalfont{\textbf{Proof of \autoref{growththeorem4}~\autoref{growththeorem4_part2}}}}}] The Yang--Mills--Higgs energy $\sE_4(A,\Phi)$ is finite implies that the energy density $e(A,\Phi)$ is in $L^1( \mathbb R^4)$.
This together with the estimate in \autoref{estimate1}, satisfied by $e(A,\Phi)$, allows us to apply Heinz trick from \autoref{Heinz} by taking $f=e(A,\Phi)$. 

We aim to show that for any $\alpha\in (0,1)$, 
\begin{equation}\label{eq r alpha4d}\abs{\Phi}=O(r^{\alpha})\ \ \text{as}\ \  r=\abs{x}\to \infty.\end{equation}
By \autoref{Bochner1}, $\abs{\Phi}^2$ is subharmonic. Therefore,  there exists a point $x_0$ on $\partial B_r(0)$ such that
 \[M\coloneqq \abs{\Phi(x_0)}^2=\sup_{x\in B_r(0)} \abs{\Phi(x)}^2.\] 
Applying the inequality governing the Sobolev embedding $W^{1, \frac 4{1-\alpha}} (\mathbb R^4)\hookrightarrow C^{0,\alpha}(\mathbb R^4)$ (a consequence of Morrey's inequality), and then using Kato's inequality, we obtain
	\begin{align*}\abs{\Phi(x_0)}^2-\abs{\Phi(0)}^2  \lesssim r^{\alpha}  \|\nabla \abs{\Phi}^2\|_{L^{\frac{4}{1-\alpha}}\big(B_r(0)\big)} &\lesssim r^{\alpha} \sqrt M  \norm{\nabla_A\Phi}_{L^{\frac{4}{1-\alpha}}(B_r(0))}\\
	&\lesssim r^{\alpha} \sqrt M  \norm{e(A,\Phi)}^{1/2}_{L^{\frac 2{1-\alpha}}(\mathbb R^4)}.\end{align*}
By \autoref{Heinz}~\autoref{Heinz_part2} with $f=e(A,\Phi)$ and by Young's inequality with any $\delta>0$ satisfying $\delta \|e\|_{L^1{(\mathbb R^4)}}<1$, we obtain
 \[ r^{\alpha} \sqrt M  \norm{e(A,\Phi)}^{1/2}_{L^{\frac 2{1-\alpha}}(\mathbb R^4)}\lesssim \delta^{-1} r^{2\alpha}+ \delta M \norm{e(A,\Phi)}_{L^1{(\mathbb R^4)}}.\]
 Hence, $$M\lesssim \abs{\Phi(0)}^2+ r^{2\alpha}.$$ This proves the equation \autoref{eq r alpha4d}.

Let $G$ be the Green's kernel on $\mathbb R^4$. Set,
  $$\psi(x):=-\int_{\mathbb R^4}G(x,\cdot )\Delta \abs{\Phi}^2=2\int_{\mathbb R^4}G(x,\cdot) \big(\abs{\nabla_A\Phi}^2+2\abs {\mu(\Phi) }^2\big), \quad x\in \mathbb R^4.$$  
Then $\psi(x)$ exists and $\psi:\mathbb R^4\to [0,\infty)$ is a smooth function satisfying
\[\Delta \psi= 2\abs{\nabla_A\Phi}^2+4\abs {\mu(\Phi) }^2, \quad \psi=o(1)\ \ \text{as}\ \  r=\abs{x}\to \infty.\]
 The proof can be found in \cite[Lemma 2.10]{Fadel2022}. For clarity, we note the correspondence of notation used therein:
\[ n=4, \quad X=\mathbb R^4, \quad f=2\abs{\nabla_A\Phi}^2+4\abs {\mu(\Phi) }^2\in L^1(\mathbb R^4)\cap L^3(\mathbb R^4)\cap C^\infty(\mathbb R^4). \]

Since $\abs{\Phi}^2+\psi$ is harmonic and $\abs{\Phi}^2+\psi\geq0$, by the gradient estimate for harmonic functions and by the equation \autoref{eq r alpha4d}, we obtain that $\abs{\Phi}^2+\psi$ is constant, say $m$. This finishes the proof.
\end{proof}

\section{Generalized Seiberg--Witten Bogomolny equations in dimension three}\label{section GSWB monopole}
Our objective is to establish analogous \autoref{growththeorem3} for generalized Seiberg--Witten (GSW) Bogomolny monopoles in three dimensions to those previously obtained for GSW monopoles in four dimensions. To that end, we begin by laying the necessary groundwork on the generalized Seiberg--Witten Bogomolny equations in dimension three. This includes again introducing the fundamental setup, clarifying the relevant notations, and deriving several key identities that will play a crucial role in the arguments to follow.
\subsection{Preliminaries: basic set up and identities}
We review the generalized Seiberg--Witten (GSW) Bogomolny equations in dimension three. All the constructions are similar to the GSW equations in dimension four as described in \autoref{sec GSW 4d}.  

Choose an algebraic data $(H,\rho,G)$ as in in \autoref{sec GSW 4d}. 
Set $$\Spin^H(3):= \frac{\Sp(1)\times H}{\{\pm 1\}}.$$ The group $\Sp(1)$ acts on $\mathbb R^3\cong \Im\mathbb H$ by $p\cdot x=p x\bar{p}$ and yields a $2$-fold covering $\Sp(1)\to \SO(3)$ and therefore $\Spin(3)=\Sp(1)$. 
\begin{definition}
	A \textbf{$\Spin^H$-structure} on an oriented Riemannian $3$-manifold $(M,g)$ is a principal $\Spin^H(3)$-bundle $\mathfrak s$ together with an isomorphism  \[\mathfrak s \times_{\Spin^H(3)} \SO(3)\cong \SO(TM).\qedhere\]
\end{definition}

Choose an algebraic data $(H,\rho,G)$. A $\Spin^H$-structure $\mathfrak s$ induces the following associated bundles and maps:
\begin{itemize}
	\item The \textbf{spinor bundle}, $$\bS=\mathfrak s \times_{\Spin^H(3)}S,$$
 \item The \textbf{adjoint bundle} and the \textbf{auxiliary bundle}, respectively, 
 $$\ad(\mathfrak s):=\mathfrak s \times_{\Spin^H(3)} \mathfrak g \qandq \sK:=\mathfrak s \times_{\Spin^H(3)} K,$$
	\item The \textbf{Clifford multiplication map} $\gamma: TM \to \End(\mathbf S,\mathbf S)$ induced by $\gamma$,
 \item $\tilde \gamma:TM\tn \ad(\fs)\to \End(S,S)$ is induced by $\tilde \gamma$,
     \item The \textbf{moment map} $\mu:\mathbf S \to \Lambda^2T^*M\otimes \ad(\fs),$ defined by
	 $$\mu(\Phi):=\frac 12\tilde \gamma^*(\Phi \Phi^*).$$  
\end{itemize}

\begin{definition}
	A \textbf{geometric data} is a tuple $(M,g,\fs,B)$ where
	$\mathfrak s$ is a $\Spin^H$-structure  on an oriented Riemannian $3$-manifold $(M,g)$ and $B$ is a connection on the auxiliary bundle $\sK$.
\end{definition}

Choose a geometric data $(M,g,\fs,B)$. Denote by $\sA(\fs,B)$ the space of all connections on $\fs$ inducing the Levi-Civita connection on $TM$ and the connection $B$ on the auxiliary bundle $\sK$. For $A\in \sA(\fs,B)$ we denote the induced connection on $\ad(\fs)$ by $\ad(A)$. Note that $\sA(\fs,B)$ is nonempty and is an affine space over $\Omega^1(M,\ad(\fs))$. Every $A\in \sA(\fs,B)$ defines a \textbf{Dirac operator} 
$\slD_A: \Gamma(\mathbf S)\to \Gamma(\mathbf S)$ which is given by 
$$ \slD_A\Phi=\sum_{i=1}^{3} \gamma(e_i)\nabla_{A,{e_i}}\Phi,$$
	where $ \{ e_1,e_2,e_3\}$ is an oriented local orthonormal frame of $TM$.

\begin{definition}
	The \textbf{generalized Seiberg--Witten (GSW) Bogomolny equations in dimension three} associated with the datas $(H,\rho,G)$ and $(M,g,\fs,B)$ are the following equations: for $A\in\sA(\fs,B) $, $\xi\in \Omega^0(M,\ad(\fs))$, $\Phi \in  \Gamma(\mathbf S)$,
	\begin{equation} \label{GSWb eq}
	\slD_A\Phi=-\rho(\xi)\Phi,\ \ \ 
	F_{\ad(A)}=*d_{\ad(A)}\xi+\mu(\Phi).  
	\end{equation}
 Solutions of \autoref{GSWb eq} are said to be \textbf{generalized Seiberg--Witten (GSW) Bogomolny monopoles}. With $\xi=0$, \autoref{GSWb eq} is called \textbf{generalized Seiberg--Witten (GSW) equations} and the solutions are called \textbf{generalized Seiberg--Witten (GSW) monopoles}.
\end{definition}
\begin{remark}
 Choose an algebraic data $(H,\rho,G)$ and a geometric data $(M,g,\fs,B)$ in dimension three. We consider the four manifold $X:=\mathbb R \times M$ with the cylindrical metric $dt^2+ g$. Let $\pi:\mathbb R \times M\to M$ be  the standard projection onto $M$. The $\Spin^H$-structure $\mathfrak s$ on $M$ will induce a $\Spin^H$-structure on $X$, again call it by $\mathfrak s$, under the inclusion $\Sp(1)\hookrightarrow \Sp(1)\times \Sp(1)$ defined by $x\to(x,x)$, and subsequently positive/negative spinor bundles $\mathbf S^\pm$. Auxiliary bundle on $X$ is the pull back of the auxiliary bundle on $M$ and we take the connection on the auxiliary bundle is the pullback connection of $B$. Both $S^\pm$ are identified with $\pi^*\mathbf S$. Let $\Phi \in \Gamma(\mathbf S)$, $A\in \sA(\fs,B)$ over $M$ and $\xi \in \Omega^0(M, \ad(\fs))$. Then $A,\xi $ will induce a connection $\bA \in \sA(\fs,\pi^*B)$ over $X$ such that $\ad(\bA)=\pi^*\ad(A)+\pi^*\xi \ dt$. Consider $\pi^*\Phi  \in\Gamma (\pi^*\mathbf S) \cong \Gamma(\mathbf S^+)$.  Then the equations \autoref{sw eq} on $X$ for $(\bA,\pi^*\Phi)$ under the identifications above are equivalent to the equations \autoref{GSWb eq} on $M$ for $(A,\xi,\Phi)$. Thus the dimensional reduction of the GSW equations on $\mathbb R\times M$ is the GSW Bogomolny equations on $M$.
\end{remark}

\begin{definition}\label{def auxiliary curvature 3d} We define the \textbf{auxiliary curvature operator} $\fR\in \End(\bS)$ by 
\[\fR:=\frac{\text{scal}_g}{4} + \tilde\gamma(F_B).\qedhere\]
\end{definition}
\begin{prop}[{Lichenerowicz--Weitzenb\"ock formula, \cite[Proposition 5.1.5]{Morgan1996}}]\label{prop bochner weitzen 3}
	Suppose $A\in\sA(\fs,B)$ and $\Phi \in\Gamma(\mathbf S)$. Then
	\begin{equation*}
	\slD_A^2\Phi=\nabla^*_A\nabla_A\Phi+\tilde\gamma(F_{\ad(A)})\Phi+ \fR\Phi.
	\end{equation*}
\end{prop}

The following identities, whose proofs can be found in \cite[Appendix B]{Doan2017a}, will be useful in later sections. 
\begin{prop}\label{prop identity  3}
	Suppose $\xi\in\Omega^0(M,\ad(\fs)), a \in\Omega^1(M,\ad(\fs)),$ and $\Phi\in\Gamma(\mathbf S)$. Then
	\begin{enumerate}[(i)]
 \item $[\xi,\mu(\Phi)]=2\mu(\Phi,\rho(\xi)\Phi)$,
	\item $[a\wedge \mu(\Phi)]=-*\rho^*((\tilde \gamma(a)\Phi)\Phi^*)$.
	\end{enumerate}
\end{prop}
\begin{prop}\label{prop identity with d, d* dim 3}
	Suppose $A\in\mathcal A(\fs,B)$ and $\Phi\in\Gamma(\bS)$. Then
	\begin{enumerate}[(i)]
 \item $d_{\ad(A)}\mu(\Phi)=-*\rho^*((\slD_A\Phi)\Phi^*)$,
\item $d^*_{\ad(A)}\mu(\Phi)=*2\mu(\slD_A\Phi,\Phi)-\rho^*((\nabla_A\Phi)\Phi^*)$.
\end{enumerate}
\end{prop}
\begin{remark}
	Suppose M is an oriented closed Riemannian $3$-manifold and $(A,\xi,\Phi)$ is a solution of the GSW Bogomolny equations \autoref{GSWb eq}. Then $\nabla_{\ad(A)} \xi=0$, $\rho(\xi)\Phi=0$ and $(A,\Phi)$ is a GSW monopole. Indeed, by Bianchi identity and \autoref{prop identity with d, d* dim 3} we get
 \[0\leq \int_M\inp{\xi}{\Delta_{\ad(A)}\xi}=-\int_M\inp{\xi}{*d_{\ad(A)}\mu(\Phi)}=\int_M\inp{\xi}{\rho^*((\slD_A\Phi)\Phi^*)}=-\int_M\abs{\rho(\xi)\Phi}^2.\qedhere\]
\end{remark}
We again define a Yang--Mills--Higgs energy (YMH) functional and will see in the following that on an oriented closed $3$-manifold absolute minima of this functional are generalized Seiberg--Witten Bogomolny monopoles.  
\begin{definition}The \textbf{Yang--Mills--Higgs energy functional} $\sE_3:\sA(\fs,B) \times \Omega^0(M,\ad(\fs)) \times \Gamma(\mathbf S)\to \mathbb R$ is defined by
	\[\sE_3(A,\xi,\Phi)=
\|F_{\ad(A)}\|_{L^2}^2+\|\nabla_{\ad(A)}\xi\|_{L^2}^2+\|\nabla_A\Phi\|_{L^2}^2+\|\mu(\Phi)\|_{L^2}^2+\|\rho(\xi)\Phi\|_{L^2}^2+\inp{\fR \Phi}{\Phi}_{L^2}.\qedhere\]
\end{definition}

\begin{remark}\label{benergy equal}
	Suppose M is an oriented closed Riemannian $3$-manifold. Then by \autoref{prop bochner weitzen 3}, \autoref{prop identity with d, d* dim 3} and Bianchi idenity we obtain for any $ (A,\xi,\Phi) \in \sA(\fs,B) \times \Omega^0(M,\ad(\fs)) \times \Gamma(\mathbf S)$, 
	\begin{align*}
&\int_M\abs{\slD_A\Phi+\rho(\xi)\Phi}^2+
	\abs{F_{\ad(A)}-*d_{\ad(A)}\xi-\mu(\Phi)}^2\\	&=\sE_3(A,\xi,\Phi)+2\int_M\inp{\slD_A\Phi}{\rho(\xi)\Phi}+\inp{*d_{\ad(A)}\mu(\Phi)}{\xi}=\sE_3(A,\xi,\Phi).
	\end{align*}
Thus the absolute minima of $\sE_3$ are GSW Bogomolny monopoles. \end{remark}
\begin{prop}
	The \textbf{Euler--Lagrange equations} of the energy functional $\sE_3$ are the following:
	\begin{equation}\label{EL eq}
	\begin{split}
	d^*_{\ad(A)} F_{\ad(A)}=[d_{\ad(A)}\xi,\xi]-\rho^*((\nabla_A\Phi)\Phi^*),\\
	\Delta_{\ad(A)}\xi=-\rho^*((\rho(\xi)\Phi)\Phi^*),\\
	\nabla_A^*\nabla_A\Phi=\rho(\xi)^2\Phi-\tilde \gamma(\mu(\Phi))\Phi- \fR\Phi.
	\end{split}
	\end{equation}
\end{prop}
\begin{proof}Suppose $A \in\sA(\fs,B)$, $a\in \Omega^1(M,\ad(\fs))$, $\xi,\eta\in \Omega^0(M,\ad(\fs))$, $\Phi,\Psi \in  \Gamma(\mathbf S)$. Assume that $a$, $\eta$, $\Psi$ are compactly supported. For $\abs{t}\ll 1$ we obtain,
\[\frac d {dt}\|F_{\ad(A)+ta}\|_{L^2}^2=2\inp{d_A^*F_{\ad(A)}}{a}_{L^2}+O(t),\]
\[\frac d {dt}\|\nabla_{\ad(A)+ta}(\xi+t\eta)\|_{L^2}^2=2\inp{\Delta_{\ad(A)}\xi}{\eta}_{L^2}-2\inp{[d_{\ad(A)}\xi,\xi]}{a}_{L^2}+O(t),\]
\[\frac d {dt}\|\nabla_{A+ta}(\Phi+t\Psi)\|_{L^2}^2=2\inp{\nabla_A^*\nabla_A\Phi}{\Psi}_{L^2}+2\inp{\rho^*((\nabla_A\Phi)\Phi^*)}{a}_{L^2}+O(t),\]
\[\frac d {dt}\|\rho(\xi+t\eta)(\Phi+t\Psi)\|_{L^2}^2=2\inp{\rho^*((\rho(\xi)\Phi)\Phi^*)}{\eta}_{L^2}-2\inp{\rho(\xi)^2\Phi}{\Psi}_{L^2}+O(t),\]
\[\frac d {dt}\|\mu(\Phi+t\Psi)\|_{L^2}^2=2\inp{\tilde \gamma(*\mu(\Phi))\Phi}{\Psi}_{L^2}+O(t),\]
and
\[\frac d {dt}\inp{\fR (\Phi+t\Psi)}{\Phi+t\Psi}_{L^2}=2\inp{\fR\Phi}{\Psi}_{L^2}+O(t).\qedhere\]
\end{proof}

\begin{remark}
	If $ (A,\xi,\Phi)$ is a GSW Bogomolny monopole, then it satisfies the Euler--Lagrange equations \autoref{EL eq}. Indeed, this follows from \autoref{benergy equal} directly.
	Alternatively we can do  the following direct computations using \autoref{prop identity with d, d* dim 3} and \autoref{prop identity  3}.
\begin{align*}
		d^*_{\ad(A)} F_{\ad(A)} &=*d_{\ad(A)} d_{\ad(A)}\xi+d^*_{\ad(A)}\mu(\Phi)\\
		&=*[F_{\ad(A)},\xi]+ 2*\mu(\slD_A\Phi,\Phi)-\rho^*((\nabla_A\Phi)\Phi^*)\\
		&= [d_{\ad(A)} \xi,\xi]+*[\mu(\Phi),\xi]- 2*\mu(\rho(\xi)\Phi,\Phi)-\rho^*((\nabla_A\Phi)\Phi^*)\\
  &=[d_{\ad(A)}\xi,\xi]-\rho^*((\nabla_A\Phi)\Phi^*),
		\end{align*}
		\begin{align*}
		\Delta_{\ad(A)}\xi=&d^*_{\ad(A)} d_{\ad(A)}\xi=-*d_{\ad(A)}\mu(\Phi)=\rho^*((\slD_A\Phi)\Phi^*)=-\rho^*((\rho(\xi)\Phi)\Phi^*),
		\end{align*}
	\begin{align*}
		\nabla_A^*\nabla_A\Phi&=-\slD_A(\rho(\xi)\Phi)-\tilde \gamma(F_{\ad(A)})\Phi- \fR\Phi\\
		&=-\rho(\xi)\slD_A\Phi+\tilde \gamma(*d_{\ad(A)}\xi)\Phi-\tilde \gamma(*d_{\ad(A)}\xi)\Phi-\tilde \gamma(\mu(\Phi))\Phi- \fR\Phi\\
		&=\rho(\xi)^2\Phi-\tilde \gamma(\mu(\Phi))\Phi- \fR\Phi.\qedhere
		\end{align*}
\end{remark}
By taking inner product with $\xi$ and $\Phi$ in the second and third equations of the equations \autoref{EL eq}, we derive the following Bochner identities as a corollary.
\begin{cor}\label{Bochner 3}
	Let $(A, \xi, \Phi)$ be a solution to the generalized Seiberg--Witten Bogomolny equations \autoref{eq GSWb 3d formal}, or more generally, to the Euler--Lagrange equations \autoref{EL eq}. Then
	\begin{equation}\label{eq Bochner 3}
	\frac 12\Delta\abs{\Phi}^2+\abs{\rho(\xi)\Phi}^2+2\abs{	\mu(\Phi)}^2+\abs{\nabla_A\Phi}^2+\inp{\fR\Phi}{\Phi}=0,
	\end{equation}
and
\begin{equation*}
	\frac 12\Delta\abs{\xi}^2+\abs{\rho(\xi)\Phi}^2+\abs{\nabla_{\ad(A)}\xi}^2=0. \qedhere
	\end{equation*}
\end{cor}
The following corollary is obtained by applying an integration by parts to the above Bochner identities.
\begin{cor}
	Let $\Omega$ be a bounded open subset of $X$ with smooth boundary $\partial \Omega$  and $f\in C^\infty(\bar \Omega)$. Suppose $(A,\xi,\Phi)$ satisfies the equations \autoref{eq Bochner 3} on $\Omega$, then 
	\begin{align*}&\frac 12 \int_\Omega \Delta f \cdot \abs{\Phi}^2+\int_\Omega f \cdot (\abs{\nabla_A\Phi}^2+\abs{\rho(\xi)\Phi}^2+2\abs {\mu(\Phi) }^2)\\
	&=-\int_\Omega f \cdot \inp{\fR\Phi}{\Phi}
	+\frac 12 \int_{\partial \Omega} f\cdot \partial_\nu\abs{\Phi}^2-\partial_\nu f \cdot\abs{\Phi}^2.\qedhere
	\end{align*}
\end{cor}

The next two propositions highlight how the maximum principle imposes significant restrictions on the behavior of the GSW Bogomolny monopoles under the assumption of non-negative self-dual auxiliary curvature. 

\begin{prop} Let $(M,g)$ be an oriented Riemannian $3$-manifold and $(A,\xi,\Phi)$ be a GSW Bogomolny monopole or more generally a solution of the Euler--Lagrange equations \autoref{EL eq}. Then
	\begin{enumerate}[(i)]
		\item If  M is closed, then $\abs{\xi}^2$ is constant, or equivalently $\rho(\xi)\Phi=0$ and $\nabla_{\ad(A)}\xi= 0$.
		\item If M is noncompact and $\abs{\xi}^2$ {decays to zero at infinity}. Then $ \xi=0$.\qedhere
		\end{enumerate}
  \end{prop}
  \begin{proof}
	Since $\frac 12\Delta\abs{\xi}^2=-\abs{\rho(\xi)\Phi}^2-\abs{\nabla_{\ad(A)}\xi}^2\leq 0$,  $\abs{\xi}^2$ is subharmonic. This implies the required assertions after applying the maximum principle.
\end{proof}
\begin{prop}
	Let $(M,g)$ be an oriented Riemannian $3$-manifold and $\fR\geq0$ (i.e. $\inp{\fR\Phi}{\Phi}\geq 0 \  \forall \Phi\in \Gamma(\bS)$. Let $(A,\xi,\Phi)$ be a a GSW Bogomolny monopole or more generally a solution of the Euler--Lagrange equations \autoref{EL eq}. Then 
		\begin{enumerate}[(i)]
		\item If  M is closed,  then $\abs{\Phi}$ is constant, or equivalently $\rho(\xi)\Phi=0$, $\nabla_A\Phi= 0$, $\mu(\Phi)=0$ and $\inp{\fR\Phi}{\Phi}=0$.
	\item If M is noncompact and $\abs{\Phi}^2$ {decays to zero at infinity}. Then $ \Phi=0$.\qedhere
		\end{enumerate}
\end{prop}
\begin{proof}
	Since $\frac 12 \Delta \abs{\Phi}^2= -\abs{\nabla_A\Phi}^2-\abs{\rho(\xi)\Phi}^2-2\abs {\mu(\Phi) }^2-\inp{\fR\Phi}{\Phi}\leq 0$,  $\abs{\Phi}^2$ is subharmonic. This implies the required assertions after applying the maximum principle.
\end{proof}
\subsection{Frequency function and the proof of \autoref{growththeorem3} (\ref{growththeorem3_part1})}
Throughout this subsection, we impose the following standing assumption, which is a part of \autoref{growththeorem3}.
\begin{hypothesis}\label{hyp main 3d} $M=\mathbb R^3$ with the standard Euclidean metric and orientation, and the auxiliary connection $B$ is chosen so that the auxiliary curvature operator $\fR=\tilde \gamma(F_B)\in \End(\bS)$ vanishes.
\end{hypothesis}

 Let $(A, \xi, \Phi)$ be a solution to the generalized Seiberg--Witten Bogomolny equations \autoref{eq GSWb 3d formal}, or more generally, to the Euler--Lagrange equations \autoref{EL eq} associated with the Yang--Mills--Higgs energy functional $\sE_3$.  Denote by $r$ the radial distance function from the origin in $\mathbb R^3$. \autoref{growththeorem3}~\autoref{growththeorem3_part1} concerns the asymptotic behavior of the $L^2$-norm of $\Phi$ averaged over spheres of radius $r$ as $r \to \infty$. To investigate this behavior, we employ the frequency function approach as discussed in \autoref{frequency_4d} adapted to three dimensions. Our treatment is again closely follows the approach in \cite{Walpuski2019}.

\begin{definition} Denote by $B_{r}$  the open ball in $\mathbb R^3$ centered at $0$. 
	\begin{enumerate}[i)]
		\item For every $r>0$ we define 
		$$m(r):=\frac 1{r^2}\int_{\partial B_r}\abs{\Phi}^2\ \ \  \text{and}\ \ \ D(r):=\frac 1{r}\int_{B_r}\abs{\nabla_A \Phi}^2+2\abs{\mu(\Phi)}^2+\abs{\rho(\xi)\Phi}^2.$$
		\item Set $r_{-1}:= \sup \{0,r: r\in (0,\infty):m(r)=0\}$. The \textbf{frequency function} $N: (r_{-1},\infty)\to [0,\infty)$ is defined by 
		\[N(r):=\frac{D(r)}{m(r)}=\frac {r\int_{B_r}\abs{\nabla_A \Phi}^2+2\abs{\mu(\Phi)}^2+\abs{\rho(\xi)\Phi}^2 }{\int_{\partial B_r}\abs{\Phi}^2}.\qedhere\]
	\end{enumerate}
\end{definition}
Our objective is again to analyze the monotonicity behavior of $N(r)$ , and for that, we need to compute its derivative, $N^\prime(r)$. To begin, we first calculate the derivative of the squared $L^2$-norm average of $\Phi$, $m(r)$ as follows:
\begin{prop}\label{prop m prime 3d}For every $r>0$,
\[m^\prime(r)=\frac {2D(r)}r.\]
\end{prop}
\begin{proof}The proof is again a direct computation.
	\begin{align*}
	m^\prime(r)&= \frac 1{ r^2} \frac {d}{dr}\int_{\partial B_r}\abs{\Phi}^2-\frac 2{r^3} \int_{\partial B_r}\abs{\Phi}^2\\
	&=\frac 1{r^2} (\int_{\partial B_r}\frac 2r\abs{\Phi}^2+ \int_{\partial B_r}\partial_r\abs{\Phi}^2)-\frac 2{r^3} \int_{\partial B_r}\abs{\Phi}^2=\frac 2r D(r).\qedhere
	\end{align*}
\end{proof}
\begin{cor}\label{cor m prime 3d}We have
\begin{enumerate}[a)]
		\item $m^\prime(r)\geq0, \forall r\in (0,\infty)$, and if $\Phi\neq 0$ then $r_{-1}=0$,
		\item for every $r\in (r_{-1},\infty)$, \[m^\prime(r)=\frac {2N(r)}r m(r).\qedhere\]
	\end{enumerate}
\end{cor}
Since $N(r)$ is the quotient of $D(r)$ and $m(r)$, we must also compute the derivative of $D(r)$. The following proposition provides the derivative:
\begin{prop}\label{Dprime3d}For every $r>0$,
	\[D^\prime(r)+\frac 1{r}\int_{\partial B_r}\abs{F_{\ad(A)}}^2+\abs{\nabla_{\ad(A)}\xi}^2-\abs{\mu(\Phi)}^2+\frac 1{r^2}\int_{B_r}\abs{F_{\ad(A)}}^2-\abs{\mu(\Phi)}^2\geq 0.\]
	\end{prop}
	To prove this proposition, we require a lemma about the divergence-free property of a certain symmetric $(0,2)$ tensor field $T$, similar to the approach by \citet[Proof of Lemma 5.2]{Taubes2012}.
\begin{definition}\label{def T3d}	
	The symmetric $(0,2)$ tensor $T$ is defined by $T:=T_1+T_{2}+T_3+T_4$ where
\begin{align*}T_1(v,w)=\inp{\nabla_{A,v}\Phi}{\nabla_{A,w}\Phi}-\frac 12\inp{v}{w}\abs{\nabla_A \Phi}^2 ,\\
	T_2(v,w)=\inp{\iota_vF_{\ad(A)}}{\iota_wF_{\ad(A)}}-\frac12 \inp{v}{w}\abs{F_{\ad(A)}}^2,\\
		T_3(v,w)=\inp{\nabla_{\ad(A),v}\xi}{\nabla_{\ad(A),w}\xi}-\frac 12\inp{v}{w}\abs{\nabla_{\ad(A)} \xi}^2, \\
		T_4(v,w)=	-\frac12 \inp{v}{w}\abs{\rho(\xi)\Phi}^2-\frac12 \inp{v}{w}\abs{\mu(\Phi)}^2.
  \end{align*}
  Note that $2\tr (T)=-\abs{\nabla_A \Phi}^2+\abs{F_{\ad(A)}}^2-\abs{\nabla_{\ad(A)} \xi}^2-3\abs{\rho(\xi)\Phi}^2-3\abs{\mu(\Phi)}^2.$
\end{definition}

\begin{lemma}\label{lem3d div T is 0}
The divergence of $T$ is given by
	\[\nabla^*T=0.\]
\end{lemma}
\begin{proof}
	Let $p\in \mathbb R^3$ and $\{e_i\}$ be an oriented orthonormal frame around $p$ such that $\nabla_{e_i}e_j(p)=0$.
	\begin{align*}
	(\nabla^*T_1)(e_i)
	&=-\sum_j\inp{\nabla_j \nabla_j\Phi}{\nabla_i\Phi}+\inp{\nabla_j\Phi}{\nabla_j \nabla_i\Phi}-\inp{\nabla_j\Phi}{\nabla_i \nabla_j\Phi}\\
	&= \inp{\nabla_A^*\nabla_A\Phi}{\nabla_i\Phi}+\sum_j\inp{\nabla_j\Phi}{F_{\ad(A)}(e_i,e_j)\Phi}\\
	&=\inp{\rho(\xi)^2\Phi}{\nabla_i\Phi}-\inp{\tilde \gamma(\mu(\Phi))\Phi }{\nabla_i\Phi}+\sum_j\inp{\nabla_j\Phi}{\rho(F_{\ad(A)}(e_i,e_j))\Phi}\\
	&=\inp{\rho(\xi)^2\Phi}{\nabla_i\Phi}-\inp{\mu(\Phi)}{\nabla_{ad(A),e_i}\mu(\Phi)}+\sum_j\inp{\rho^*((\nabla_j\Phi)\Phi^*)}{F_{\ad(A)}(e_i,e_j)}\\
	&=\inp{\rho(\xi)\Phi}{\rho(\nabla_i\xi)\Phi}-\frac 12\nabla_i\abs{\rho(\xi)\Phi}^2-\frac 12\nabla_i\abs{\mu(\Phi)}^2+\inp{\rho^*((\nabla_A\Phi)\Phi^*)}{\iota_{e_i}F_{\ad(A)}},
	\end{align*}
		\begin{align*}
	&(\nabla^*T_2)(e_i)\\
	&=-\sum_j\inp{\nabla_j\iota_{e_i}F_{\ad(A)}}{\iota_{e_j}F_{\ad(A)}}+\inp{\nabla_j\iota_{e_j}F_{\ad(A)}}{\iota_{e_i}F_{\ad(A)}}+\frac 12 \nabla_i\abs{F_{\ad(A)}}^2\\
	&=-\sum_j\inp{{e_j}\w\iota_{e_i}\nabla_jF_{\ad(A)}}{F_{\ad(A)}}+\inp{\iota_{e_j}\nabla_jF_{\ad(A)}}{\iota_{e_i}F_{\ad(A)}}+\frac 12\nabla_i\abs{F_{\ad(A)}}^2\\
	&=\sum_j\inp{\iota_{e_i}{e_j}\w\nabla_jF_{\ad(A)}}{F_{\ad(A)}}-\inp{\nabla_iF_{\ad(A)}}{F_{\ad(A)}}+\inp{d_A^*F_{\ad(A)}}{\iota_{e_i}F_{\ad(A)}}+\frac 12\nabla_i\abs{F_{\ad(A)}}^2\\
	&=\sum_j\inp{\iota_{e_i}d_{\ad(A)} F_{\ad(A)}}{F_{\ad(A)}}+\inp{d_{\ad(A)}^*F_{\ad(A)}}{\iota_{e_i}F_{\ad(A)}}\\
	&=\inp{[d_{\ad(A)}\xi,\xi]}{\iota_{e_i}F_{\ad(A)}}-\inp{\rho^*((\nabla_A\Phi)\Phi^*)}{\iota_{e_i}F_{\ad(A)}},
\end{align*}
and
	\begin{align*}
(\nabla^*T_3)(e_i)
&=-\sum_j\inp{\nabla_j \nabla_j\xi}{\nabla_i\xi}+\inp{\nabla_j\xi}{\nabla_j \nabla_i\xi}-\inp{\nabla_j\xi}{\nabla_i \nabla_j\xi}\\
&= \inp{\nabla_{\ad(A)}^*\nabla_{\ad(A)}\xi}{\nabla_i\xi}+\sum_j\inp{\nabla_j\xi}{[F_{\ad(A)}(e_i,e_j),\xi]}\\
&=- \inp{\rho^*((\rho(\xi)\Phi)\Phi^*)}{\nabla_i\xi}-\inp{[d_{\ad(A)}\xi,\xi]}{\iota_{e_i}F_{\ad(A)}}\\
&=-\inp{\rho(\xi)\Phi}{\rho(\nabla_i\xi)\Phi}-\inp{[d_{\ad(A)}\xi,\xi]}{\iota_{e_i}F_{\ad(A)}}.
\end{align*}
Since $\nabla^*T_4(e_i)=\frac 12\nabla_i\abs{\rho(\xi)\Phi}^2+\frac 12\nabla_i\abs{\mu(\Phi)}^2$, we have $\nabla^*T=0$.
\end{proof}

\begin{proof}[{Proof of \autoref{Dprime3d}}]
	We have $$D^\prime(r) =-\frac 1{r} D(r)+ \frac 1{r}\int_{\partial B_r}\abs{\nabla_A \Phi}^2+2\abs{\mu(\Phi)}^2+\abs{\rho(\xi)\Phi}^2 .$$ Now
	\begin{align*} 
	0&=\int_{B_r}\inp{\nabla^*T}{dr^2}\\
	&=-2r\int_{\partial B_r}T(\partial_r,\partial_r)+\int_{B_r}2\tr(T)\\
	&=-2r\int_{\partial B_r}\abs{\nabla_{A,\partial_r} \Phi}^2+\abs{\iota(\partial_r)F_{\ad(A)}}^2+\abs{\nabla_{\ad (A),\partial_r} \xi}^2\\
	&+r\int_{\partial B_r}\abs{\nabla_A \Phi}^2+\abs{F_{\ad(A)}}^2+\abs{\nabla_{\ad(A)}\xi}^2+\abs{\mu(\Phi)}^2+\abs{\rho(\xi)\Phi}^2 +\int_{B_r}2\tr(T)\\
	&=-2r\int_{\partial B_r}\abs{\nabla_{A,\partial_r} \Phi}^2+\abs{\iota(\partial_r)F_{\ad(A)}}^2+\abs{\nabla_{\ad(A),\partial_r} \xi}^2+r^2D^\prime(r)\\
	& +r\int_{\partial B_r}\abs{F_{\ad(A)}}^2+\abs{\nabla_{\ad(A)}\xi}^2-\abs{\mu(\Phi)}^2 +\int_{B_r}\abs{F_{\ad(A)}}^2-\abs{\nabla_{\ad(A)} \xi}^2-2\abs{\rho(\xi)\Phi}^2-\abs{\mu(\Phi)}^2.\qedhere
	\end{align*}
\end{proof}
A computation analogous to the one used earlier in the proof of \autoref{prop N prime } yields the following derivative estimate of $N(r)$.
\begin{prop}\label{prop N prime 3d}For all $r>r_{-1}$ we have
\[N^\prime(r)+\frac 1{rm(r)}\int_{\partial B_r}\abs{F_{\ad(A)}}^2+\abs{\nabla_A\xi}^2-\abs{\mu(\Phi)}^2+\frac 1{r^2m(r)}\int_{B_r}\abs{F_{\ad(A)}}^2-\abs{\mu(\Phi)}^2\geq 0.\]
\end{prop}

\begin{remark}
\label{remark almostmonotonicity3d}
From the above proposition, it is evident that $N(r)$ may not exhibit monotonicity. However, if $(A, \Phi)$ is a solution to the generalized Seiberg--Witten equations \autoref{eq GSW 3d formal}, then $ F_{\ad(A)}=\mu(\Phi)$. In this case, by using \autoref{prop N prime  3d}, we obtain the inequality
\begin{equation}\label{eqnalmostmonotonesw3d}N^\prime(r) \geq 0.\end{equation}
 Otherwise, we obtain instead the inequality
\begin{equation}\label{eqnalmostmonotoneEL3d}N^\prime(r)+\frac 1{rm(r)}\int_{\partial B_r}\abs{F_{\ad(A)}}^2+\abs{\nabla_A\xi}^2+\frac 1{r^2m(r)}\int_{B_r}\abs{F_{\ad(A)}}^2\geq 0.\end{equation}
For the later case, we accordingly modify the frequency function $N(r)$ in the following proof, ensuring that it exhibits the necessary monotonicity, provided we are given the assumptions in  \autoref{growththeorem3}~\autoref{growththeorem3_part1}. 
\end{remark}

\begin{proof}[{{\normalfont{\textbf{Proof of \autoref{growththeorem3}~\autoref{growththeorem3_part1}}}}}]
Assume $\Phi\neq 0$. By \autoref{cor m prime 3d}, $r_{-1}=0$. Evidently, $N(r)=0 \ \forall r>0$ if and only if $\nabla_A\Phi=0$, $\rho(\xi)\Phi=0$ and $\mu(\Phi)=0$, or equivalently, by \autoref{Bochner 3}, $\abs{\Phi}$ is constant. Assume now on that $N\neq0$. 
	 
First consider the case, when $\xi=0$ and $(A, \Phi)$ solves  the equation \autoref{eq GSW 3d formal}. Then the inequality \autoref{eqnalmostmonotonesw3d} in \autoref{remark almostmonotonicity3d}  implies that $N^\prime>0$. Since $N\neq 0$, there exists $s>0$ such that $N(s)>0$. Set $\epsilon:=2N(s)$. Therefore for all $t\in[s,r]$ we have
	$$\frac {2N(s)}t\leq\frac{d}{dt}\log(m(t))=\frac {2N(t)}t \leq\frac {2 N(r)}t.$$
This implies 
		$$\Big(\frac rs\Big)^{2N(s)}m(s)\leq m(r)\leq \Big(\frac rs\Big)^{2N(r)}m(s).$$
 Hence \[\liminf_{r\to\infty}\frac 1{r^{\varepsilon}} m(r)\gtrsim \frac{m(s)}{s^\varepsilon}>0.\]
 Observe that, the assumptions mentioned in \autoref{growththeorem3_part1} of the theorem are not required for this case.
 
Now consider the general case under the assumptions of  \autoref{growththeorem3}~\autoref{growththeorem3_part1}. From the inequality \autoref{eqnalmostmonotoneEL3d} in \autoref{remark almostmonotonicity3d}, we obtain that for every $c>0$ there exists $\rho>0$ such that 
\begin{equation*}N^\prime(r)+\frac {c}{r^{2}m(\rho)}\geq0, \forall r\geq\rho.\end{equation*}
Note that the frequency function may not be monotone in this case. Define the modified frequency function 
\[\widetilde N_c(r):=N(r)-\frac c{m(\rho)r}, \quad \forall r\geq\rho.\] 
It follows that ${\widetilde N_c(r)}^\prime\geq 0, \forall r\geq\rho$, which gives the desired almost monotonicity property. This $\widetilde N_c$ controls $m$ as in the proof of \autoref{growththeorem4}~\autoref{growththeorem4_part1}. In particular, we obtain that there exists $c>0$ such that $\widetilde N_c(s)>0$ for some $s\geq\rho$. Denoting $\varepsilon:=2\widetilde N_c(s)>0$ we have 
	$$m(r)\geq r^\varepsilon \frac{m(s)}{s^\varepsilon}\Big(e^{\int_s^r\frac {c}{m(\rho)t^{2}} dt}\Big).$$
	Thus \[\liminf_{r\to\infty}\frac 1{r^{\varepsilon}} m(r)\gtrsim \frac{m(s)}{s^\varepsilon}>0.\]
	This completes the proof.
\end{proof}

\subsection{Consequence of finite energy and the proof of \autoref{growththeorem3} (\ref{growththeorem3_part2})}
 In this section, we also assume \autoref{hyp main 3d}. Let $(A, \xi, \Phi)$ be a solution to the generalized Seiberg--Witten Bogomolny equations \autoref{eq GSWb 3d formal}, or more generally, to the Euler--Lagrange equations \autoref{EL eq} associated with the Yang--Mills--Higgs energy functional $\sE_3$. We will show that if $\sE_3(A,\xi,\Phi)$ is finite, then $\xi$ and $\abs{\Phi}$ must converge to non-negative constants $m_1$  and $m_2$ respectively at infinity. The key idea is to once again apply Heinz trick ($\varepsilon$-regularity) from \autoref{Heinz} to the energy density $e(A, \xi, \Phi)$, which serves as the integrand in the Yang--Mills--Higgs energy functional$\sE_3$. The proof follows a line of reasoning similar to that in \autoref{finite_energy4}, which itself  draws on several related arguments from  \cites{Nagy2019,Fadel2022}.
\begin{definition}
The \textbf{ energy density function} $e: \sA(\fs,B) \times \Omega^0(\mathbb R^3,\ad(\fs)) \times \Gamma(\mathbf S)\to  C^\infty(\mathbb R^3,\mathbb R)$ is defined by
\[e(A,\xi,\Phi)=
|F_{\ad(A)}|^2+|\nabla_{\ad(A)}\xi|^2+|\nabla_A\Phi|^2+|\mu(\Phi)|^2+|\rho(\xi)\Phi|^2.\qedhere\]
\end{definition}

\begin{lemma}\label{estimate}Suppose $(A, \xi, \Phi)$ is a solution to the generalized Seiberg--Witten Bogomolny equations \autoref{eq GSWb 3d formal}, or more generally, to the Euler--Lagrange equations \autoref{EL eq}. Then
	$$\Delta e(A,\xi,\Phi) \lesssim e(A,\xi,\Phi)+e(A,\xi,\Phi)^{\frac 32}.$$
\end{lemma}
\begin{proof}The proof is similar to the proof of \autoref{estimate1}.
We are going to use Lichenerowicz--Weitzenb\"ock formula for Lie-algebra bundle valued $1$ and $2$-forms, Euler--Lagrange equations \autoref{EL eq} and \autoref{prop identity with d, d* dim 3}.

	\begin{align*}
   \frac 12\Delta \abs{F_{\ad(A)}}^2
   &\leq \inp{\nabla_{\ad(A)}^*\nabla_{\ad(A)} F_{\ad(A)}}{F_{\ad(A)}}\\
	&\lesssim \inp{\Delta_{\ad(A)} F_{\ad(A)}}{F_{\ad(A)}}+\abs{F_{\ad(A)}}^2+\abs{F_{\ad(A)}}^3\\
	&=\inp{d_{\ad(A)}[d_{\ad(A)}\xi,\xi]-d_{\ad(A)}\rho^*((\nabla_A\Phi)\Phi^*)}{F_{\ad(A)}}+\abs{F_{\ad(A)}}^2+\abs{F_{\ad(A)}}^3\\
	&=\inp{[[F_{\ad(A)},\xi],\xi]-[d_{\ad(A)}\xi\w d_{\ad(A)}\xi]-\rho^*((\rho(F_{\ad(A)})\Phi)\Phi^*)}{F_{\ad(A)}}\\
	&-\inp{\rho^*(\nabla_A\Phi \w(\nabla_A\Phi)^*)}{F_{\ad(A)}}+\abs{F_{\ad(A)}}^2+\abs{F_{\ad(A)}}^3 \\
	&\lesssim -\abs{[F_{\ad(A)},\xi]}^2+\abs{d_A\xi}^2\abs{F_{\ad(A)}}-\abs{\rho(F_{\ad(A)})\Phi}^2+\abs{\nabla_A\Phi}^2\abs{F_{\ad(A)}}\\&+\abs{F_{\ad(A)}}^2+\abs{F_{\ad(A)}}^3 \lesssim e(A,\xi,\Phi)+e(A,\xi,\Phi)^{\frac 32},
	\end{align*}
	\begin{align*}
\frac 12\Delta|\nabla_{\ad(A)}\xi|^2
	&\leq \inp{\nabla_{\ad(A)}^*\nabla_{\ad(A)} \nabla_{\ad(A)}\xi}{\nabla_{\ad(A)}\xi}\\
	&\lesssim\inp{\Delta_{\ad(A)} d_{\ad(A)}\xi}{\nabla_{\ad(A)}\xi}+\abs{\nabla_{\ad(A)}\xi}^2+\abs{F_{\ad(A)}}\abs{\nabla_{\ad(A)}\xi}^2\\
	&= \inp{d_{\ad(A)}^* [F_{\ad(A)},\xi]-d_{\ad(A)}\rho^*((\rho(\xi)\Phi)\Phi^*)}{\nabla_{\ad(A)}\xi}+(1+\abs{F_{\ad(A)}})\abs{\nabla_{\ad(A)}\xi}^2\\
	&\lesssim \inp{[[d_{\ad(A)}\xi,\xi]-\rho^*((\nabla_A\Phi)\Phi^*),\xi]}{\nabla_{\ad(A)}\xi}-|\rho(\nabla_{\ad(A)}\xi)\Phi|^2\\	
	&+|\rho(\xi)\Phi|\abs{\nabla_{\ad(A)}\xi}\abs{\nabla_A\Phi}+(1+\abs{F_{\ad(A)}})\abs{\nabla_{\ad(A)}\xi}^2\\
	&\lesssim e(A,\xi,\Phi)+e(A,\xi,\Phi)^{\frac 32},
	\end{align*}
		\begin{align*}
	\frac 12\Delta|\nabla_A\Phi|^2
	&\leq \inp{\nabla_A^*\nabla_A\nabla_A\Phi}{\nabla_A\Phi}\\
	&=\inp{[\nabla_A^*\nabla_A,\nabla_A]\Phi}{\nabla_A\Phi}+\inp{\nabla_A\nabla_A^*\nabla_A\Phi}{\nabla_A\Phi} \\
	&\lesssim \inp{\rho(d_{\ad(A)}^*F_{\ad(A)})\Phi}{\nabla_A\Phi}+\abs{F_{\ad(A)}}\abs{\nabla_A\Phi}^2+\inp{\nabla_A(\rho(\xi)^2\Phi-\tilde \gamma(\mu(\Phi))\Phi)}{\nabla_A\Phi} \\
	&\lesssim \abs{\nabla_{\ad(A)}\xi}\abs{\rho(\xi)\Phi}\abs{\nabla_A\Phi}-\abs{\rho^*(\nabla_A\Phi)\Phi^*}^2+\abs{F_{\ad(A)}}\abs{\nabla_A\Phi}^2\\
	& -\abs{\rho(\xi)\nabla_A\Phi}^2 - 2\abs{\mu(\nabla_A\Phi,\Phi)}^2+\abs{\mu(\Phi)}\abs{\nabla_A\Phi}^2\lesssim e(A,\xi,\Phi)+e(A,\xi,\Phi)^{\frac 32},
		\end{align*}
\begin{align*}
\frac 12\Delta|\mu(\Phi)|^2
&\leq \inp{\nabla_{\ad(A)}^*\nabla_{\ad(A)}\mu(\Phi)}{\mu(\Phi)}\\
&=2\inp{\mu(\nabla_A^*\nabla_A\Phi,\Phi)-\mu(\nabla_A\Phi,\nabla_A\Phi)}{\mu(\Phi)}\\
&=2\inp{\mu(\rho(\xi)^2\Phi-\tilde \gamma(\mu(\Phi))\Phi,\Phi)-\mu(\nabla_A\Phi,\nabla_A\Phi)}{\mu(\Phi)}\\
&\lesssim\abs{\rho(\xi)\Phi}^2\abs{\mu(\Phi)}+\abs{\mu(\Phi)}^3+\abs{\nabla_A\Phi}^2\abs{\mu(\Phi)}\lesssim e(A,\xi,\Phi)^{\frac 32},
\end{align*}
and
\begin{align*}
\frac 12\Delta|\rho(\xi)\Phi|^2
&\leq  \inp{\nabla_A^*\nabla_A(\rho(\xi)\Phi)}{\rho(\xi)\Phi}\\
&\lesssim \inp{\rho(\Delta_{\ad(A)}\xi)\Phi}{\rho(\xi)\Phi}+\abs{\nabla_A\xi}\abs{\nabla_A\Phi}\abs{\rho(\xi)\Phi}+\inp{\rho(\xi)(\nabla_A^*\nabla_A\Phi)}{\rho(\xi)\Phi}\\
&\lesssim -\inp{\rho^*((\rho(\xi)\Phi)\Phi^*)\Phi}{\rho^*(\rho(\xi)\Phi)}+\abs{\nabla_A\xi}\abs{\nabla_A\Phi}\abs{\rho(\xi)\Phi}\\
&+\inp{\rho(\xi)(\rho(\xi)^2\Phi-\tilde \gamma(\mu(\Phi))\Phi)}{\rho(\xi)\Phi}\\
&\lesssim -\abs{\rho^*(\rho(\xi)\Phi)\Phi^*}^2+\abs{\nabla_A\xi}\abs{\nabla_A\Phi}\abs{\rho(\xi)\Phi}+\abs{\rho(\xi)\Phi}^2\abs{\mu(\Phi)}-\abs{\rho(\xi)^2\Phi}^2\\
&\lesssim e(A,\xi,\Phi)^{\frac 32}.\qedhere
\end{align*}
\end{proof}
\begin{proof}[{{\normalfont{\textbf{Proof of \autoref{growththeorem3}~\autoref{growththeorem3_part2}}}}}] The Yang--Mills--Higgs energy $\sE_3(A,\xi,\Phi)$ is finite implies that the energy density $e(A,\xi,\Phi)$ is in $L^1( \mathbb R^3)$.
This together with the estimate in \autoref{estimate}, satisfied by $e(A,\xi,\Phi)$, allows us to apply Heinz trick from \autoref{Heinz} by taking $f=e(A,\xi,\Phi)$. 

We aim to show that for any $\alpha\in (0,1)$, 
\begin{equation}\label{eq r alpha}\abs{\xi}=O(r^{\alpha}),\quad \abs{\Phi}=O(r^{\alpha})\ \ \text{as}\ \  r=\abs{x}\to \infty.\end{equation}
By \autoref{Bochner 3}, $\abs{\xi}^2$ and $\abs{\Phi}^2$ are subharmonic. The proof of the above claim for both $\xi$ and $\Phi$ are exactly as in the proof of \autoref{growththeorem3}~\autoref{growththeorem3_part2}, with the only difference now is the inequality governing the Sobolev embedding $W^{1, \frac 3{1-\alpha}} (\mathbb R^3)\hookrightarrow C^{0,\alpha}(\mathbb R^3)$. 

Let $G$ be the Green's kernel on $\mathbb R^3$. Set,
   $$\psi_1(x):=-\int_{\mathbb R^3}G(x,\cdot )\Delta \abs{\xi}^2=2\int_{\mathbb R^3}G(x,\cdot) \big(\abs{\rho(\xi)\Phi}^2+\abs{\nabla_{\ad(A)}\xi}^2\big), \quad x\in \mathbb R^3,$$  
  and 
  $$\psi_2(x):=-\int_{\mathbb R^3}G(x,\cdot )\Delta \abs{\Phi}^2=2\int_{\mathbb R^3}G(x,\cdot) \big(\abs{\rho(\xi)\Phi}^ 2+2\abs{	\mu(\Phi)}^2+\abs{\nabla_A\Phi}^2\big), \quad x\in \mathbb R^3.$$
Again by \autoref{Heinz}~\autoref{Heinz_part2} and \cite[Lemma 2.10]{Fadel2022}, we obtain that $\psi_i(x)$ exists and $\psi_i:\mathbb R^3\to [0,\infty)$ is a smooth function satisfying
\[\psi_i=o(1)\ \ \text{as}\ \  r=\abs{x}\to \infty, \quad i=1,2.\]

Since $\abs{\xi}^2+\psi_1$ and $\abs{\Phi}^2+\psi_2$ are harmonic, exactly same reason as in the proof of \autoref{growththeorem3}~\autoref{growththeorem3_part2} implies that $\abs{\xi}^2+\psi_1$ and $\abs{\Phi}^2+\psi_2$ are constants, say $m_1$ and $m_2$, respectively. This finishes the proof.
\end{proof}

\appendix
\section{Heinz trick and $\varepsilon$-regularity}
\begin{lemma}\label{Heinz}
	Let $f:\mathbb R^n\to[0,\infty), n\leq 4$ be a smooth function satisfying $$\Delta f \lesssim f+ f^{3/2}.$$
	Then 
\begin{enumerate}	
	\item \label{Heinz_part1}
	there exist constants $\varepsilon_0, C_0>0$ such that for any $0<r<1$ and any point $x\in \mathbb R^n$ for which
	\[r^{4-n}\int_{B_r(x)}f < \varepsilon_0,\]
	we have the estimate \begin{equation}\label{eqHeinzpointwise} \sup_{y\in B_{ r/4}(x) } f(y)\leq  \frac {C_0}{r^n}{\int_{B_r(x)}f}.\end{equation}
	\item \label{Heinz_part2}
	if $f\in L^1( \mathbb R^n)$, then
	$f= o(1), \ \text{as } r \to \infty.$
	Moreover, there is a constant $C_f>0$ depending on $f$, such that for any $1\leq p\leq \infty$, 
	\[ \norm{f}_{L^p(\mathbb R^n)}\leq C_f \norm{f}_{L^1(\mathbb R^n)}.\]
		\end{enumerate}
\end{lemma}
\begin{proof}The proof of \autoref{Heinz_part1} can be found in \cite[Lemma A.1]{Walpuski2015a}. For clarity, we note the correspondence of notation used therein:
\[U=B_1(x), \quad d=4, \quad p=1, \quad q\coloneqq\frac 2d+1=\frac 32, \quad \delta=0. \]

The proof of \autoref{Heinz_part2} follows from the arguments in \cite[Proposition 3.1, Corollary 3.2]{Nagy2019} and in \cite[Corollary 4.4]{Fadel2022}. For the reader’s convenience, we include the proof below.  Since $f\in L^1( \mathbb R^n)$, given the constant $\varepsilon_0>0$ from \autoref{Heinz_part1}, there exists $R>0$ such that 
\[\int_{\mathbb R^n\setminus B_R(0)}f< \varepsilon_0.\]
Since $n\leq 4$, the condition in \autoref{Heinz_part1} is satisfied for all $x\in \mathbb R^n\setminus B_{R+2}(0)$. Therefore by taking $r=\frac 12$ in the inequality \autoref{eqHeinzpointwise}, we obtain
\[ f (x) \leq 2^n \cdot C_0 {\int_{B_{1/2}(x)}f}.\]
Since $f\in L^1( \mathbb R^n)$, the integral in the right-hand side tends to zero at infinity, that is of $o(1)$ as $r=\abs{x} \to \infty$; and consequently,  so does $f(x)$. This implies that $f\in {L^\infty(\mathbb R^n)} \cap {L^1(\mathbb R^n)}$ and there exists $x_*\in \mathbb R^n$ where $f$ attains its maximum. By choosing $r_*\in (0,1)$ small enough such that $r_*^{4-n}\int_{B_{r_*}(x_*)}f < \varepsilon_0$, the inequality \autoref{eqHeinzpointwise} yields the estimate 
\[\norm{f}_{L^\infty(\mathbb R^n)}\leq C_0 r_*^{-n} \norm{f}_{L^1(\mathbb R^n)}.\] 
To derive the estimate in \autoref{Heinz_part2} for any $1<p< \infty$, we apply the H\"older inequality:
\[ \norm{f}_{L^p(\mathbb R^n)}\leq \norm{f}^{(p-1)/p}_{L^\infty(\mathbb R^n)}\norm{f}^{1/p}_{L^1(\mathbb R^n)}\leq (C_0 r_*^{-n})^{(p-1)/p}  \cdot \norm{f}_{L^1(\mathbb R^n)}.\]
By choosing $C_f:= \max\{ 1, C_0 r_*^{-n}\}$ we obtain the required estimate.
\end{proof}

\printreferences


@article{Fadel2022,
	author = {Fadel, Daniel},
	day = {14},
	doi = {10.1007/s12220-022-01095-8},
	issn = {1559-002X},
	journal = {The Journal of Geometric Analysis},
	month = {Nov},
	number = {1},
	pages = {17},
	title = {Asymptotics of Finite Energy Monopoles on AC 3-Manifolds},
	url = {https://link.springer.com/content/pdf/10.1007/s12220-022-01095-8.pdf},
	volume = {33},
	year = {2022}}

@article{Bleher23,
	author = {Michael Bleher},
	arxiv = {2306.17017},
	month = {06},
	title = {Growth of the Higgs Field for Kapustin-Witten solutions on ALE and ALF gravitational instantons},
	url = {https://arxiv.org/pdf/2306.17017.pdf},
	year = {2023}}

@article{Hitchin1982,
	author = {Hitchin, N. J.},
	day = {01},
	doi = {10.1007/BF01208717},
	issn = {1432-0916},
	journal = {Communications in Mathematical Physics},
	month = {Dec},
	number = {4},
	pages = {579--602},
	title = {Monopoles and geodesics},
	url = {https://link.springer.com/content/pdf/10.1007/BF01208717.pdf},
	volume = {83},
	year = {1982}}

@article{Nagy2019,
	author = {A. Nagy and G. Oliveira},
	doi = {10.1007/s11005-021-01426-w},
	arxiv = {1906.05435},
	journal = {Letters in Mathematical Physics, 111, Issue 4, Article: 87 (2021)},
	month = {06},
	title = {The Kapustin--Witten equations on ALE and ALF gravitational instantons},
	url = {https://arxiv.org/pdf/1906.05435.pdf},
	year = {2019}}

@article{Nagy2020,
	author = {A. Nagy and G. Oliveira},
	doi = {10.1007/s00029-020-00584-4},
	arxiv = {1906.05432},
	journal = {Selecta Mathematica,},
	pages = {58},
	title = {The Haydys monopole equation},
	url = {https://arxiv.org/pdf/1906.05432.pdf},
	volume = {26},
	year = {2020}}

@book{Mares2010,
	author = {Mares, B},
	note = {Thesis (Ph. D.)--Massachusetts Institute of Technology},
	title = {Some analytic aspects of Vafa-Witten twisted $N = 4$ supersymmetric Yang-Millseory theory},
	url = {http://hdl.handle.net/1721.1/64488},
	year = {2010}}

@incollection{Almgren1979,
	author = {Almgren, Jr., F. J.},
	booktitle = {Minimal submanifolds and geodesics},
	mr = {574247},
	pages = {1--6},
	publisher = {North-Holland, Amsterdam-New York},
	title = {Dirichlet's problem for multiple valued functions and the regularity of mass minimizing integral currents},
	year = {1979},
	zbl = {0439.49028}}

@article{Bryan1996,
	author = {Bryan, J. A. and Wentworth, R.},
	journal = {{Turkish Journal of Mathematics}},
	mr = {1392667},
	number = {1},
	pages = {119--128},
	title = {{The multi-monopole equations for K{\"a}hler surfaces}},
	volume = {20},
	year = {1996},
	zbl = {0873.53049}}

@book{Walpuski2022,
  title     = {Lectures on generalised Seiberg–Witten equations, Frontiers in Geometry and Topology},
  author = {Walpuski,T},
  editor    = {Paul M. N. Feehan and Lenhard L. Ng and Peter S. Ozsváth},
  series    = {Proceedings of Symposia in Pure Mathematics},
  volume    = {109},
  year      = {2024},
  publisher = {American Mathematical Society},
  address   = {Providence, RI},
  isbn      = {978-1-4704-7087-6},
  doi       = {10.1090/pspum/109}
}

@article{Doan2017a,
	arxiv = {1704.02954},
	author = {Doan, A. and Walpuski, T.},
	doi = {10.1007/s00029-020-00574-6},
	journal = {Selecta Mathematica},
	number = {3},
	title = {{Deformation theory of the blown-up Seiberg--Witten equation in dimension three}},
	url = {https://walpu.ski/Research/SeibergWittenDeformationTheory.pdf},
	volume = {26},
	year = 2020,
	zbl = {1446.53043}}

@article{Doan2017d,
	arxiv = {1712.08383},
	author = {Doan, A. and Walpuski, T.},
	doi = {10.4310/PAMQ.2019.v15.n4.a4},
	journal = {Pure and Applied Mathematics Quarterly},
	number = {4},
	pages = {1047--1133},
	title = {{On counting associative submanifolds and Seiberg--Witten monopoles}},
	url = {https://walpu.ski/Research/CountingAssociativesSeibergWitten.pdf},
	volume = {15},
	year = {2019}}

@book{Donaldson1990,
	address = {New York},
	author = {Donaldson, S.~K. and Kronheimer, P.~B.},
	mr = {MR1079726},
	series = {Oxford Mathematical Monographs},
	title = {The geometry of four-manifolds},
	year = {1990},
	zbl = {0904.57001}}

@article{Feehan1998,
	author = {Feehan, P. M. N. and Leness, T. G.},
	issn = {0022-040X},
	journal = {J. Differential Geom.},
	mr = {1664908},
	mrclass = {57R57 (58D27)},
	mrreviewer = {Liviu I. Nicolaescu},
	number = {2},
	pages = {265--410},
	title = {{$\PU(2)$ monopoles. I. Regularity, Uhlenbeck compactness, and transversality}},
	url = {http://projecteuclid.org/euclid.jdg/1214461020},
	volume = {49},
	year = {1998}}

@article{Haydys2008,
	arxiv = {0706.0389},
	author = {Haydys, A.},
	coden = {CMPHAY},
	doi = {10.1007/s00220-008-0445-1},
	fjournal = {Communications in Mathematical Physics},
	issn = {0010-3616},
	journal = {Communications in Mathematical Physics},
	mr = {MR2403610},
	mrclass = {30G35 (35Q40 58Jxx)},
	number = {1},
	pages = {251--261},
	title = {Nonlinear {D}irac operator and quaternionic analysis},
	volume = {281},
	year = {2008},
	zbl = {1230.30034}}

@misc{Haydys2017,
	arxiv = {1703.06329},
	author = {Haydys,A.},
	title = {{$\Gtwo$ instantons and the Seiberg--Witten monopoles}},
	year = {2017}}

@book{Jaffe1980,
	author = {Jaffe, A. and Taubes, C.H.},
	mr = {614447},
	publisher = {Birkh{\"a}user},
	series = {{Progress in Physics}},
	subtitle = {Structure of static gauge theories},
	title = {Vortices and monopoles},
	volume = {2},
	year = {1980},
	zbl = {0457.53034}}

@article{Kapustin2007,
	arxiv = {hep-th/0604151},
	author = {Kapustin, A. and Witten, E.},
	issn = {1931-4523},
	journal = {Communications in Number Theory and Physics},
	mr = {2306566},
	number = {1},
	pages = {1--236},
	title = {{Electric-magnetic duality and the geometric Langlands program}},
	volume = {1},
	year = {2007},
	zbl = {1128.22013}}

@book{Morgan1996,
	author = {Morgan, J.~W.},
	mr = {1367507},
	number = {44},
	publisher = {Princeton University Press, Princeton, NJ},
	series = {Mathematical Notes},
	title = {The {S}eiberg--{W}itten equations and applications to the topology of smooth four-manifolds},
	year = {1996},
	zbl = {0846.57001}}

@article{Nakajima2015,
	arxiv = {1503.03676},
	author = {Nakajima, H.},
	doi = {10.4310/ATMP.2016.v20.n3.a4},
	journal = {Advances in Theoretical and Mathematical Physics},
	mr = {3565863},
	number = {3},
	pages = {595--669},
	title = {{Towards a mathematical definition of Coulomb branches of $3$--dimensional $\mathcal N=4$ gauge theories, I}},
	volume = {20},
	year = {2016},
	zbl = {06659459}}

@article{Okonek1996,
	arxiv = {alg-geom/9505029},
	author = {Okonek, C. and Teleman, A.},
	issn = {0010-3616},
	journal = {Communications in Mathematical Physics},
	mr = {1405956},
	number = {2},
	pages = {363--388},
	title = {Quaternionic monopoles},
	url = {http://projecteuclid.org/euclid.cmp/1104287353},
	volume = {180},
	year = {1996}}

@incollection{Pidstrygach2004,
	author = {Pidstrygach, V. Ya.},
	booktitle = {{Algebraic geometry. Methods, relations, and applications. Collected papers. Dedicated to the memory of Andrei Nikolaevich Tyurin.}},
	mr = {2101297},
	mrclass = {53C26 (57R57)},
	pages = {249--262},
	publisher = {Moscow: Maik Nauka/Interperiodica},
	title = {{Hyper-K{\"a}hler manifolds and Seiberg--Witten equations}},
	year = {2004},
	zbl = {1101.53026}}

@article{Seiberg1994,
	author = {Seiberg, N. and Witten, E.},
	journal = {{Nuclear Physics. B}},
	number = {1},
	pages = {19--52},
	title = {{Electric-magnetic duality, monopole condensation, and confinement in $N=2$ supersymmetric Yang--Mills theory.}},
	volume = {426},
	year = {1994},
	zbl = {0996.81510}}

@incollection{Taubes1999b,
	author = {Taubes, C.H.},
	booktitle = {Trends in mathematical physics},
	mr = {1708781},
	mrclass = {58J60 (53C26 53C27 57R57 58J05)},
	mrreviewer = {Yi-Jen Lee},
	pages = {475--486},
	publisher = {American Mathematical Society},
	series = {AMS/IP Studies in Advanced Mathematics},
	title = {{Nonlinear generalizations of a $3$--manifold's Dirac operator}},
	volume = {13},
	year = {1999},
	zbl = {1049.58504}}

@article{Taubes2012,
	arxiv = {1205.0514},
	author = {Taubes, C.H.},
	doi = {10.4310/CJM.2013.v1.n2.a2},
	issn = {2168-0930; 2168-0949/e},
	journal = {{Cambridge Journal of Mathematics}},
	mr = {3272050},
	msc2010 = {53C05 58J 53C07 53C21},
	number = 2,
	pages = {239--397},
	publisher = {International Press of Boston, Inc.},
	title = {{$\PSL(2;\C)$ connections on 3--manifolds with $L^2$ bounds on curvature.}},
	volume = 1,
	year = 2013,
	zbl = {1296.53051}}

@misc{Taubes2017a,
	arxiv = {1701.03072},
	author = {Taubes, C.H.},
	title = {{Growth of the Higgs field for solutions to the Kapustin-Witten equations on $\R^4$}},
	year = 2017}

@misc{Taubes2013,
	arxiv = {1307.6447},
	author = {Taubes, C.H.},
	title = {{Compactness theorems for $\SL(2; \C)$ generalizations of the $4$--dimensional anti-self dual equations}},
	year = {2013}}

@misc{Taubes2017,
	arxiv = {1702.04610},
	author = {Taubes, C.H.},
	title = {{The behavior of sequences of solutions to the Vafa--Witten equations}},
	year = {2017}}

@article{Vafa1994,
	author = {Vafa, C. and Witten, E.},
	doi = {10.1016/0550-3213(94)90097-3},
	issn = {0550-3213},
	journal = {Nuclear Physics. B},
	mr = {MR1305096},
	number = {1-2},
	pages = {3--77},
	title = {{A strong coupling test of $S$--duality}},
	volume = 431,
	year = 1994,
	zbl = {0964.81522}}

@article{Walpuski2015a,
	arxiv = {1507.03258},
	author = {Walpuski, T.},
	doi = {10.4171/CMH/423},
	journal = {Commentarii Mathematici Helvetici},
	mr = {3623252},
	number = {4},
	pages = {751--776},
	title = {{A compactness theorem for Fueter sections}},
	url = {https://walpu.ski/Research/FueterCompactness.pdf},
	volume = {92},
	year = {2017},
	zbl = {1383.58009}}

@article{Walpuski2019,
	arxiv = {1904.03749},
	author = {Walpuski, T. and Zhang, B.},
	doi = {10.1215/00127094-2021-0005},
	journal = {Duke Mathematical Journal},
	number = {17},
	publisher = {Duke University Press},
	title = {On the compactness problem for a family of generalized Seiberg{\textendash}Witten equations in dimension 3},
	url = {https://walpu.ski/Research/3DSeibergWittenCompactness.pdf},
	volume = {170},
	year = {2021}}

@incollection{Witten2012,
	arxiv = {1108.3103},
	author = {Witten, E.},
	booktitle = {{Proceedings of the Freedman Fest}},
	pages = {291--308},
	title = {{Khovanov homology and gauge theory.}},
	year = {2012},
	zbl = {1276.57013}}

@article{Witten2015,
	arxiv = {1506.04293},
	author = {Witten, E.},
	doi = {10.1016/j.aim.2017.06.021},
	journal = {Advances in Mathematics},
	mr = {3762001},
	pages = {624--707},
	title = {{More On Gauge Theory And Geometric Langlands}},
	volume = {327},
	year = {2018},
	zbl = {1392.81188}}
\end{document}